\documentclass[12pt,reqno]{amsart}
\usepackage[cp1251]{inputenc}
\usepackage[T2A]{fontenc}
\usepackage[russian,english]{babel}
\usepackage{geometry}
\usepackage{amsmath, amsthm, amscd, amsfonts, amssymb, graphicx, color}
\usepackage[bookmarksnumbered, colorlinks, plainpages]{hyperref}

\numberwithin{equation}{section}

\newtheorem{theorem}{Theorem}[section]

\newtheorem{lemma}[theorem]{Lemma}
\newtheorem{proposition}[theorem]{Proposition}

\theoremstyle{remark}
\newtheorem{remark}[theorem]{Remark}

\theoremstyle{definition}

\newtheorem*{main-definition}{Main Definition}

\begin{document}

\title[Elliptic boundary-value problems]{Elliptic boundary-value problems\\ in some distribution spaces\\ of generalized smoothness}


\author[A. Anop]{Anna Anop}

\address{Georg-August-Universit\"at G\"ottingen, Mathematisches Institut, 3-5 Bunsenstr., G\"ottingen, D-37073, Germany}

\email{anna.anop@mathematik.uni-goettingen.de, anop@imath.kiev.ua}


\author[A. Murach]{Aleksandr Murach}

\address{Institute of Mathematics of the National Academy of Sciences of Ukraine, 3 Tereshchen\-kivs'ka, Kyiv, 01024, Ukraine}

\email{murach@imath.kiev.ua}





\begin{abstract}
We build a solvability theory of elliptic boundary-value problems in normed Sobolev spaces of generalized smoothness for any integrability exponent $p>1$. The smoothness is given by a number parameter and a supplementary function parameter that varies slowly at infinity. These spaces are obtained by a combination of the methods of the complex interpolation with number parameter between Banach spaces and the quadratic interpolation with function parameter between Hilbert spaces applied to classical Sobolev spaces. We show that the spaces under study admit localization near a smooth boundary and describe their trace spaces in terms of Besov spaces with the same supplementary function parameter. We prove that a general differential elliptic problem induces Fredholm bounded operators on appropriate pairs of the spaces under study. We also find exact sufficient conditions for solutions of the problem to have a prescribed generalized or classical smoothness on a given set and establish corresponding \textit{a~priori} estimates of the solution. These results are specified for parameter-elliptic problems.
\end{abstract}

\maketitle

\section{Introduction}\label{sec1}

In the last two decades, the mathematical literature on function spaces demonstrates a growing interest in distribution spaces of generalized smoothness and their applications (see, e.g., \cite{AlmeidaCaetano11, Baghdasaryan10, CobosKuhn09, DominguezTikhonov23, FarkasLeopold06, HaroskeMoura08, HaroskeLeopoldMouraSkrzypczak23, LoosveldtNicolay19, MikhailetsMurachZinchenko25, Milatovic24, MouraNevesSchneider14}). These spaces characterize the smoothness properties of distributions by means of a function parameter depending on the frequency variables (or by a number sequence representing this function). This allows describing such properties far more finely than the scales of classical distribution spaces, which proved to be very useful in approximation theory \cite{Stepanets05}, theory of stochastic process \cite{Jacob010205}, and theory of partial differential equations (PDEs) \cite{MikhailetsMurach14, NicolaRodino10}. As to boundary-value problems for PDEs, the applications of distribution spaces of generalized smoothness has been restricted only to Hilbert spaces \cite{AnopChepurukhinaMurach21Axioms, AnopDenkMurach21, LosMikhailetsMurach17CPAA1, LosMikhailetsMurach21CPAA10, MikhailetsMurach12BJMA2, MikhailetsMurach14}, which can be investigated with the help of Spectral Theorem applied to a self-adjoint operator (see, e.g., \cite{MikhailetsMurachZinchenko25}). However,  theorems on solvability of various classes of boundary-value problems are well known for non-Hilbert distribution spaces of the H\"older or Sobolev smoothness given by a number (see, e.g., \cite{AgmonDouglisNirenberg59, LadyzhenskajaSolonnikovUraltzeva67, Lunardi95, Triebel83, Triebel95}). Concerning elliptic boundary-value problems, the expediency and prospects of their study in normed spaces of generalized smoothness were discussed by Triebel in \cite[pp. 57--60]{Triebel10}.

The purpose of this paper is to extend the classical theory of solvability of elliptic boundary-value problems in $L_{p}$-Sobolev spaces $H^{s}_{p}$ with $1<p<\infty$ (see, e.g., \cite[Chapter~V]{AgmonDouglisNirenberg59} and \cite[Chapter~5]{Triebel95}) to some Sobolev spaces of generalized smoothness.  The latter spaces form the scale $\{H^{s,\varphi}_{p}\}$ calibrated with the help of a supplementary function parameter $\varphi$ much more finely than the Sobolev scale. The positive function $\varphi(t)$ of $t\geq1$ varies slowly at infinity in the sense of Karamata, which implies that the space $H^{s,\varphi}_{p}$ is located between any spaces $H^{s-\varepsilon}_{p}$ and $H^{s+\varepsilon}_{p}$ with $\varepsilon>0$ on this scale. We study interpolation properties of the scale and show that every space $H^{s,\varphi}_{p}$ is obtained by the complex interpolation between certain spaces $H^{s_{0},\varphi_{0}}_{2}$ and $H^{s_{1}}_{p_{1}}$. Since the scale $\{H^{s,\varphi}_{2}\}$ of Hilbert spaces is obtained by the quadratic interpolation with function parameter between inner product Sobolev spaces \cite[Theorem~1.14]{MikhailetsMurach14}, we use these two interpolation methods to study the spaces $H^{s,\varphi}_{p}$ and elliptic problems in them. It is important for applications and follows from these interpolation properties that the space $H^{s,\varphi}_{p}(\Omega)$ over a bounded Euclidean domain $\Omega$ with a boundary $\Gamma\in C^{\infty}$ admits localization and that the trace operator acts continuously from whole  $H^{s,\varphi}_{p}(\Omega)$ onto the Besov space $B^{s-1/p,\varphi}_{p}(\Gamma)$ of the same supplementary generalized smoothness $\varphi$ whenever $s>1/p$. The distribution space $B^{\sigma,\varphi}_{p}(\Gamma)$ with arbitrary $\sigma\in\mathbb{R}$ also admits localization and has quite similar interpolation properties. These results are discussed in Section~\ref{sec2} of the paper.

Section~\ref{sec3} is devoted to applications of the above distribution spaces of generalized smoothness to an arbitrary differential elliptic boundary-value problem in the domain $\Omega$. We prove that the problem induces a Fredholm bounded operator between the spaces
\begin{equation}\label{b-v-problem-spaces}
H^{s,\varphi}_{p}(\Omega)\quad\mbox{and}\quad H^{s-2q,\varphi}_{p}(\Omega)\times
\prod_{j=1}^{q}B^{s-m_{j}-1/p,\varphi}_{p}(\Gamma)
\end{equation}
whenever $s>\max\{m_{1},\ldots,m_{q}\}+1/p$, with $2q$ being the even order of the elliptic equation, and $m_{1},\ldots,m_{q}$ being the orders of the boundary conditions. Studying solutions to this problem, we find sufficient (and necessary) conditions for the solution to have a prescribed generalized smoothness on a given subset of $\Omega$ and prove a corresponding \textit{a~priori} estimate of the solution. In terms of  the spaces \eqref{b-v-problem-spaces}, we also find exact sufficient  conditions for the generalized solution of the problem to be $\ell$ times continuously differentiable on the given set. The use of the function parameter $\varphi$ allows us to attain the limiting values of the number parameters determining the main smoothness of the right-hand sides of the problem.

Section~\ref{sec4} gives an application of spaces \eqref{b-v-problem-spaces} to a general parameter-elliptic boundary-value problem. We show that this problem sets an isomorphism between these spaces for sufficiently large absolute values of the parameter from the ellipticity angle of the problem. Moreover, the solutions of the problem admit a corresponding two-sided \textit{a~priori} estimate in  parameter-dependent norms with constants that do not depend on the parameter. This result is based on interpolation properties of such norms.   The results can be applied to the study of parabolic initial-boundary-value problems in some anisotropic distribution spaces of generalized smoothness (cf. \cite[Chapter~II]{AgranovichVishik64}).

Note that the solvability theory of elliptic equations and elliptic boundary-value problems in Hilbert distribution spaces of generalized smoothness is set forth in \cite{MikhailetsMurach12BJMA2, MikhailetsMurach14, MilkhailetsMurachChepurukhina23UMJ}. These spaces have various applications to the spectral theory of elliptic differential operators \cite{MikhailetsMurach24ProcEdin, MikhailetsMurachZinchenko25}.

\section{Spaces of generalized smoothness}\label{sec2}

We use distribution spaces whose generalized smoothness is given by a real number and a function parameter
 from a certain class $\Upsilon$.

By definition, $\Upsilon$ consists of all infinitely smooth functions
$\varphi:[1,\infty]\to(0,\infty)$ that satisfy the following two conditions:
\begin{itemize}
\item[(i)] $\varphi$ varies slowly at infinity in the sense of Karamata, i.e. $\varphi(\lambda t)/\varphi(t)\to1$ as $t\to\infty$ for every $\lambda>0$;
\item[(ii)] for each integer $m\geq1$, there exists a number $c_{m}>0$ such that $t^{m}|\varphi^{(m)}(t)|\leq c_{m}\varphi(t)$ whenever $t\geq1$.
\end{itemize}
Note \cite[Section 1.2, formula (1.11)]{Seneta76} that the condition $t\varphi'(t)/\varphi(t)\to0$ as $t\to\infty$ implies~(i). Of course, if $\varphi,\omega\in\Upsilon$ and $r\in\mathbb{R}$, then $\varphi\omega,\varphi/\omega,\varphi^{r}\in\Upsilon$.

An important example of a function of class $\Upsilon$ is given by an arbitrary positive function $\varphi\in C^{\infty}([1,\infty])$ of the form
\begin{equation*}
\varphi(t):=(\log t)^{r_{1}}(\log\log
t)^{r_{2}}\ldots(\underbrace{\log\ldots\log}_{k\;\mbox{\tiny times}}
t)^{r_{k}}\quad\mbox{of}\quad t\gg1,
\end{equation*}
where $1\leq k\in\mathbb{Z}$ and  $r_{1},\ldots,r_{k}\in\mathbb{R}$.

Assume that $1\leq n\in\mathbb{Z}$. Let $s\in\mathbb{R}$, $\varphi\in\Upsilon$, and $p\in(1,\infty)$. The complex linear space $H^{s,\varphi}_{p}(\mathbb{R}^{n})$ is defined to consist of all  distributions $w\in\mathcal{S}'(\mathbb{R}^{n})$ such that the distribution
\begin{equation}\label{w-s-varphi}
w_{s,\varphi}:=\mathcal{F}^{-1}
[\langle\xi\rangle^{s}\varphi(\langle\xi\rangle)(\mathcal{F}w)(\xi)]
\end{equation}
belongs to $L_{p}(\mathbb{R}^{n})$. Here, as usual, $\mathcal{S}'(\mathbb{R}^{n})$ is the linear topological space of all tempered distributions on $\mathbb{R}^{n}$; thus, $\mathcal{S}'(\mathbb{R}^{n})$ is dual to the linear topological Schwartz space $\mathcal{S}(\mathbb{R}^{n})$ of all infinitely smooth rapidly decreasing functions on $\mathbb{R}^{n}$. Moreover, $\mathcal{F}$ and $\mathcal{F}^{-1}$ denote the Fourier transform and its inverse; $\langle\xi\rangle:=(1+|\xi|^{2})^{1/2}$ is the smoothed absolute value of the vector $\xi\in\mathbb{R}^{n}$, and $L_{p}(\mathbb{R}^{n})$ is the normed space of all Lebesgue measurable functions $h:\mathbb{R}^{n}\to\mathbb{C}$ such that $|h|^{p}$ is integrable over $\mathbb{R}^{n}$. The space $H^{s,\varphi}_{p}(\mathbb{R}^{n})$ is endowed with the norm
\begin{equation*}
\|w,H^{s,\varphi}_{p}(\mathbb{R}^{n})\|:=
\|w_{s,\varphi},L_{p}(\mathbb{R}^{n})\|=
\Biggl(\;\,\int\limits_{\mathbb{R}^{n}}
|w_{s,\varphi}(x)|^{p}dx\Biggr)^{1/p}.
\end{equation*}
This space is well defined because the function $\langle\xi\rangle^{s}\varphi(\langle\xi\rangle)$ of $\xi\in\mathbb{R}^{n}$ is a pointwise multiplier on $\mathcal{S}'(\mathbb{R}^{n})$ by condition~(ii) and in view of the property $\varphi(t)\leq t$ whenever $t\gg1$, which follows from condition~(i) \cite[Section 1.5, property~$1^{\circ}$]{Seneta76}. We consider complex linear spaces throughout the paper.

The space $H^{s,\varphi}_{p}(\mathbb{R}^{n})$ belongs to the classes of Sobolev spaces of generalized smoothness introduced and investigated by Volevich and Paneach \cite[\S~13]{VolevichPaneah65}, Schechter \cite[Section~4]{Schechter67}, Goldman \cite{Goldman76}, Triebel \cite[Section~5]{Triebel77III}, Merucci \cite[Section~4]{Merucci84}, and others. This space is complete and separable and is continuously embedded in $\mathcal{S}'(\mathbb{R}^{n})$, with $\mathcal{S}(\mathbb{R}^{n})$ being  continuously and densely embedded in $H^{s,\varphi}_{p}(\mathbb{R}^{n})$. This follows from the analogous properties of $L_{p}(\mathbb{R}^{n})$ in view of the evident fact that the mapping $w\mapsto w_{s,\varphi}$ sets an isometric isomorphism between $H^{s,\varphi}_{p}(\mathbb{R}^{n})$ and $L_{p}(\mathbb{R}^{n})$. Hence, the set $C^{\infty}_{0}(\mathbb{R}^{n})$
of infinitely smooth compactly supported functions on $\mathbb{R}^{n}$ is dense in $H^{s,\varphi}_{p}(\mathbb{R}^{n})$.

If $\varphi(t)=1$ whenever $t\geq1$, then $H^{s,\varphi}_{p}(\mathbb{R}^{n})$ is the Sobolev space $H^{s}_{p}(\mathbb{R}^{n})$ with smoothness index $s$ and integrability exponent $p$. Generally, we have the dense continuous embeddings
\begin{equation}\label{ebbeddins-H}
H^{s+\varepsilon}_{p}(\mathbb{R}^{n})\hookrightarrow
H^{s,\varphi}_{p}(\mathbb{R}^{n})\hookrightarrow H^{s-\varepsilon}_{p}(\mathbb{R}^{n})\quad\mbox{for every}\;\;
\varepsilon>0.
\end{equation}
Indeed, it follows from property~(i) that $t^{-\varepsilon/2}\varphi(t)\to0$ and $t^{\varepsilon/2}\varphi(t)\to\infty$ as $t\to\infty$ \cite[Section 1.5, property~$1^{\circ}$]{Seneta76}; hence, the functions $t^{s}\varphi(t)/t^{s+\varepsilon}$ and $t^{s-\varepsilon}/t^{s}\varphi(t)$ of $t\geq1$ are integrable over $[1,\infty)$ with respect to $dt/t$, which implies \eqref{ebbeddins-H} by \cite[Subsection~4.1, Proposition~6]{Merucci84}.

These embeddings show that the function parameter $\varphi$ defines a supplementary generalized smoothness relative to the main smoothness $s$. We say that $s$ and $\varphi$ are the main and supplementary smoothness indexes, resp. Briefly saying, $\varphi$ refines the main smoothness. Therefore, the class
\begin{equation}\label{refined-H-scale}
\{H^{s,\varphi}_{p}(\mathbb{R}^{n}):s\in\mathbb{R},\varphi\in\Upsilon\}
\end{equation}
is naturally called the refined Sobolev scale (over $\mathbb{R}^{n}$)
with integrability exponent $p$. In the $p=2$ case of Hilbert spaces, such a scale was introduced by Mikhailets and Murach \cite{MikhailetsMurach05UMJ5} and applied to elliptic operators and elliptic boundary-value problems \cite{MikhailetsMurach06UMJ3, MikhailetsMurach07UMJ5, MikhailetsMurach12BJMA2, MikhailetsMurach14, Murach07UMJ6}. Various distribution spaces of generalized smoothness given by a slowly varying function were investigated by Haroske and Moura \cite{HaroskeMoura04}, specifically Besov spaces over $\mathbb{R}^{n-1}$ used below.

Let $\Omega$ be a nonempty open subset of $\mathbb{R}^{n}$, and introduce a version of scale \eqref{refined-H-scale} for~$\Omega$. By definition, the linear space $H^{s,\varphi}_{p}(\Omega)$ consists of the restrictions of all distributions $w\in H^{s,\varphi}_{p}(\mathbb{R}^{n})$ to $\Omega$ and is endowed with the norm
\begin{equation}\label{Sobolev-norm-domain}
\|u,H^{s,\varphi}_{p}(\Omega)\|:=
\inf\bigl\{\|w,H^{s,\varphi}_{p}(\mathbb{R}^{n})\|:
w\in H^{s,\varphi}_{p}(\mathbb{R}^{n}),\;w=u\;\mbox{in}\;\Omega\bigr\},
\end{equation}
where $u\in H^{s,\varphi}_{p}(\Omega)$. This space is complete and separable, is continuously embedded in $\mathcal{S}'(\Omega)$,
and $C^{\infty}_{0}(\overline{\Omega})$ is dense in it, which follows from the above-mentioned properties of $H^{s,\varphi}_{p}(\mathbb{R}^{n})$. Here, $\mathcal{S}'(\Omega)$ is the linear topological space of the restrictions of all distributions $w\in \mathcal{S}'(\mathbb{R}^{n})$ to $\Omega$, and
$C^{\infty}_{0}(\overline{\Omega})$ denotes the set of the restrictions of all functions $w\in C^{\infty}_{0}(\mathbb{R}^{n})$ to the closure $\overline{\Omega}$ of $\Omega$.

If $\varphi(\cdot)\equiv1$, then $H^{s,\varphi}_{p}(\Omega)$ becomes the Sobolev space $H^{s}_{p}(\Omega)$. The dense continuous embeddings \eqref{ebbeddins-H} remain valid if we replace $\mathbb{R}^{n}$ with $\Omega$. If $\Omega$ is bounded, these embeddings are compact, which follows from the known fact that the inequality $\ell<s$ implies the compact embedding $H^{s}_{p}(\Omega)\hookrightarrow H^{\ell}_{p}(\Omega)$ of Sobolev spaces.

Of course, $H^{s,\varphi}_{p}(\Omega)$ is the factor space of $H^{s,\varphi}_{p}(\mathbb{R}^{n})$ by the subspace
\begin{equation*}
\{w\in H^{s,\varphi}_{p}(\mathbb{R}^{n}):
\mathrm{supp}\,w\subset\mathbb{R}^{n}\setminus\Omega\}.
\end{equation*}
Hence, each $H^{s,\varphi}_{2}(\Omega)$ is a Hilbert space. It is important that this space is obtained by the quadratic interpolation (with an appropriate function parameter) between Sobolev spaces provided that
\begin{equation}\label{Omega-assumption}
\Omega=\mathbb{R}^{n}\;\;\mbox{or}\;\;
\Omega\;\,\mbox{is an open half-space}\;\;
\mbox{or}\;\;\Omega\;\,\mbox{is a bounded Lipschitz domain}
\end{equation}
(the last assumption is considered in the $n\geq2$ case).

\begin{proposition}\label{prop-quadratic-interp}
Let $s,\varepsilon,\delta\in\mathbb{R}$, $\varepsilon,\delta>0$, and $\varphi\in\Upsilon$. Define the interpolation function parameter as follows:
\begin{equation}\label{quadratic-interp-parameter}
\psi(t):=\left\{
\begin{array}{ll}
t^{\varepsilon/(\varepsilon+\delta)}\varphi(t^{1/(\varepsilon+\delta)})& \mbox{if}\;\;t\geq1 \\
1&\mbox{if}\;\;0<t<1.
\end{array}
\right.
\end{equation}
Suppose that $\Omega$ satisfies \eqref{Omega-assumption}. Then
\begin{equation}\label{quadratic-interp-Sobolev}
H^{s,\varphi}_{2}(\Omega)=
(H^{s-\varepsilon}_{2}(\Omega),H^{s+\delta}_{2}(\Omega))_{\psi}
\end{equation}
up to equivalence of norms.
\end{proposition}

Here, $(E_{0},E_{1})_{\psi}$ denotes the result of the quadratic interpolation with the function parameter $\psi$ between separable Hilbert spaces $E_{0}$ and $E_{1}$ satisfying the dense continuous embedding $E_{1}\hookrightarrow E_{0}$. The definition and properties of this interpolation are given, e.g., in \cite[Section~1.1]{MikhailetsMurach14}. We will recall this notion in Appendix~\ref{appendix}. Proposition~\ref{prop-quadratic-interp} is contained in \cite[Theorem~5.1]{MikhailetsMurach15ResMath1} (see also \cite[Theorems 1.14 and~3.2]{MikhailetsMurach14} in the case where $\Omega=\mathbb{R}^{n}$ or when the bounded domain $\Omega$ has infinitely smooth boundary).

If $p\neq2$, then the refined Sobolev scale
with integrability exponent $p$ is obtained by complex interpolation between Hilbert spaces $H^{s_0,\varphi_0}_{2}(\Omega)$ and Sobolev spaces $H^{s_1}_{p_1}(\Omega)$. Moreover, the collection of these scales is stable with respect to this interpolation. Let us formulate these properties in the form of two theorems. As usual, $[E_{0},E_{1}]_{\theta}$ denotes the result of the complex interpolation with the number parameter $\theta\in(0,1)$ between Banach spaces $E_{0}$ and $E_{1}$ continuously embedded in a Hausdorff topological linear space (see., e.g., \cite[Section~1.9]{Triebel95}). Such a pair of the spaces is called an interpolation pair.

\begin{theorem}\label{th-C-interp-Sobolev}
Let $s_{0},s_{1}\in\mathbb{R}$, $\varphi_0,\varphi_1\in\Upsilon$, $p_0,p_1\in(1,\infty)$, and $\theta\in(0,1)$. Define numbers $s$, $p$, and a function $\varphi\in\Upsilon$ by the formulas
\begin{equation}\label{C-interp-parameters}
s=(1-\theta)s_{0}+\theta s_{1},\quad
\frac{1}{p}=\frac{1-\theta}{p_0}+\frac{\theta}{p_1},\quad
\varphi(t)\equiv\varphi_{0}^{1-\theta}(t)\varphi_{1}^{\theta}(t).
\end{equation}
Suppose that $\Omega$ satisfies \eqref{Omega-assumption}. Then
\begin{equation}\label{C-interp-Sobolev}
[H^{s_0,\varphi_0}_{p_{0}}(\Omega),
H^{s_1,\varphi_1}_{p_{1}}(\Omega)]_{\theta}=
H^{s,\varphi}_{p}(\Omega)
\end{equation}
up to equivalence of the norms.
\end{theorem}

\begin{theorem}\label{th-C-interp-Sobolev-got}
Let $s\in\mathbb{R}$, $\varphi\in\Upsilon$, and $p\in(1,\infty)$ with $p\neq2$. We arbitrarily chose a real number $p_{1}$ such that $1<p_{1}<p<2$ or $2<p<p_{1}<\infty$ and define a number $\theta\in(0,1)$ by the formula
\begin{equation}\label{p-2-interp}
\frac{1}{p}=\frac{1-\theta}{2}+\frac{\theta}{p_1}.
\end{equation}
Choosing one of two real numbers $s_{0}$ and $s_{1}$ arbitrarily, we define the second number by the formula $s=(1-\theta)s_{0}+\theta s_{1}$. Finally, define a function parameter $\varphi_{0}\in\Upsilon$ by the formula $\varphi_{0}(t):=\varphi^{1/(1-\theta)}(t)$ whenever $t\geq1$. Suppose that $\Omega$ satisfies \eqref{Omega-assumption}. Then
\begin{equation}\label{C-interp-Sobolev-got}
H^{s,\varphi}_{p}(\Omega)=[H^{s_0,\varphi_0}_{2}(\Omega),
H^{s_1}_{p_1}(\Omega)]_{\theta}
\end{equation}
up to equivalence of the norms.
\end{theorem}

Note that the interpolation formulas \eqref{quadratic-interp-Sobolev}, \eqref{C-interp-Sobolev}, and \eqref{C-interp-Sobolev-got} are especially useful when we study bounded linear operators on the corresponding spaces because the boundedness of such operators is preserved under the interpolation between these spaces. Namely, if $[E_{0},E_{1}]$ and $[G_{0},G_{1}]$ are interpolation pairs of Banach spaces and if a linear mapping $T$ is a bounded operator from $E_{j}$ to $G_{j}$ for each $j\in\{0,1\}$, then $T$ is a bounded operator from $[E_{0},E_{1}]_{\theta}$ to $[G_{0},G_{1}]_{\theta}$ whenever $0<\theta<1$. The quadratic interpolation between Hilbert spaces has an analogous property.

\begin{proof}[Proof of Theorems~$\ref{th-C-interp-Sobolev}$ and~$\ref{th-C-interp-Sobolev-got}$.] Of course, the second theorem is a direct consequence of the first one. However, our proof of Theorem~\ref{th-C-interp-Sobolev} will rely on formula \eqref{C-interp-Sobolev-got}. We will therefore give a common proof of these two theorems.

If $\Omega=\mathbb{R}^{n}$, then Theorem~\ref{th-C-interp-Sobolev} is a special case of Triebel's result \cite[Theorem~4.2/2]{Triebel77III} in view of \cite[Remark~5.1]{Triebel77III}. Specifically, this result describes the complex interpolation between the general Triebel--Lizorkin spaces $F^{g(\cdot)}_{p,q}(\mathbb{R}^{n})$ introduced in \cite[Definition 2.3/4(ii)]{Triebel77III}, whereas the remark explains that $F^{g(\cdot)}_{p,2}(\mathbb{R}^{n})$ coincides with the corresponding Sobolev space of generalized smoothness. The class $M$ of function parameters $g(\cdot)$ used by Triebel \cite[Definition~2.1/2]{Triebel77III} contains the function $\langle\xi\rangle^{s}\varphi(\langle\xi\rangle)$ of $\xi\in\mathbb{R}^{n}$ whatever $s\in\mathbb{R}$ and $\varphi\in\Upsilon$. Thus, Triebel's result is applicable to the spaces appearing in the interpolation formula \eqref{C-interp-Sobolev} provided that  $\Omega=\mathbb{R}^{n}$.

If $\Omega$ is an arbitrary open domain in $\mathbb{R}^{n}$, then the restriction operator $R_{\Omega}:w\mapsto w\!\upharpoonright\!\Omega$, where $w\in\mathcal{S}'(\mathbb{R}^{n})$, acts continuously from each space $H^{s,\varphi}_{p}(\mathbb{R}^{n})$ onto $H^{s,\varphi}_{p}(\Omega)$. This implies that $R_{\Omega}$ is a bounded operator between the spaces
\begin{equation*}
R_{\Omega}:H^{s,\varphi}_{p}(\mathbb{R}^{n})=
[H^{s_0,\varphi_0}_{p_{0}}(\mathbb{R}^{n}),
H^{s_1,\varphi_1}_{p_{1}}(\mathbb{R}^{n})]_{\theta}\to
[H^{s_0,\varphi_0}_{p_{0}}(\Omega),
H^{s_1,\varphi_1}_{p_{1}}(\Omega)]_{\theta}
\end{equation*}
under the hypotheses of Theorem~\ref{th-C-interp-Sobolev}. Hence,
\begin{equation}\label{proof-interp-embed}
H^{s,\varphi}_{p}(\Omega)=R_{\Omega}(H^{s,\varphi}_{p}(\mathbb{R}^{n}))
\hookrightarrow [H^{s_0,\varphi_0}_{p_{0}}(\Omega),
H^{s_1,\varphi_1}_{p_{1}}(\Omega)]_{\theta}.
\end{equation}
The embedding of the margin Banach spaces is continuous because the norms in them are compatible. As its special case we have the continuous embedding
\begin{equation}\label{proof-interp-embed-case}
H^{s,\varphi}_{p}(\Omega)\hookrightarrow [H^{s_0,\varphi_0}_{2}(\Omega),
H^{s_1}_{p_{1}}(\Omega)]_{\theta}
\end{equation}
under the hypotheses of Theorem~\ref{th-C-interp-Sobolev-got}.

It remains to prove the inverse of \eqref{proof-interp-embed} in the case where $\Omega$ is an open half-space and in the case when $\Omega$ is a bounded Lipschitz domain. Consider these cases together. We will first substantiate the inverse of \eqref{proof-interp-embed-case}.

We use Rychkov's \cite[Theorem~4.1]{Rychkov99} universal linear extension operator for $\Omega$ and denote it by $T$. This operator extends every distribution $u\in\mathcal{S}'(\Omega)$ to $\mathbb{R}^{n}$ and acts continuously from whole $H^{s}_{p}(\Omega)$ to $H^{s}_{p}(\mathbb{R}^{n})$ whatever $s\in\mathbb{R}$ and $1<p<\infty$ (the same is true for all classical Besov and Triebel--Lizorkin normed or quasi-normed spaces). It follows from Proposition~\ref{prop-quadratic-interp}, that $T$ is a bounded operator from $H^{s,\varphi}_{2}(\Omega)$ to $H^{s,\varphi}_{2}(\mathbb{R}^{n})$ whenever $s\in\mathbb{R}$ and $\varphi\in\Upsilon$. Hence, formula \eqref{C-interp-Sobolev-got} in the $\Omega=\mathbb{R}^{n}$ case implies the bounded extension operator
\begin{equation*}
T:[H^{s_0,\varphi_0}_{2}(\Omega),H^{s_1}_{p_1}(\Omega)]_{\theta}\to
[H^{s_0,\varphi_0}_{2}(\mathbb{R}^{n}),
H^{s_1}_{p_1}(\mathbb{R}^{n})]_{\theta}=H^{s,\varphi}_{p}(\mathbb{R}^{n})
\end{equation*}
under the hypotheses of Theorem~\ref{th-C-interp-Sobolev-got}. Thus, the identity mapping $R_{\Omega}T$ sets the continuous embedding
\begin{equation*}
R_{\Omega}T:[H^{s_0,\varphi_0}_{2}(\Omega),H^{s_1}_{p_1}(\Omega)]_{\theta}
\to H^{s,\varphi}_{p}(\Omega),
\end{equation*}
which means the inverse of \eqref{proof-interp-embed-case}. The interpolation formula \eqref{C-interp-Sobolev-got} is proved.

Now we can prove the inverse of \eqref{proof-interp-embed}. It follows from \eqref{C-interp-Sobolev-got} that $T$ is a bounded operator from $H^{s,\varphi}_{p}(\Omega)$ to $H^{s,\varphi}_{p}(\mathbb{R}^{n})$ whatever $s\in\mathbb{R}$, $\varphi\in\Upsilon$, and $1<p<\infty$. This implies that $T$ is also a bounded operator on the pair of spaces
\begin{equation*}
T:[H^{s_0,\varphi_0}_{p_{0}}(\Omega),
H^{s_1,\varphi_1}_{p_{1}}(\Omega)]_{\theta}\to
[H^{s_0,\varphi_0}_{p_{0}}(\mathbb{R}^{n}),
H^{s_1,\varphi_1}_{p_{1}}(\mathbb{R}^{n})]_{\theta}=
H^{s,\varphi}_{p}(\mathbb{R}^{n})
\end{equation*}
under the hypotheses of Theorem~\ref{th-C-interp-Sobolev} used in the case where $\Omega=\mathbb{R}^{n}$. Thus, the identity mapping $R_{\Omega}T$ sets the continuous embedding
\begin{equation*}
R_{\Omega}T:[H^{s_0,\varphi_0}_{p_{0}}(\Omega),
H^{s_1,\varphi_1}_{p_{1}}(\Omega)]_{\theta}\to H^{s,\varphi}_{p}(\Omega),
\end{equation*}
which means the inverse of \eqref{proof-interp-embed}. The interpolation formula \eqref{C-interp-Sobolev} is proved.
\end{proof}

We assume henceforth that $n\geq2$ and that $\Omega$ is a bounded open domain with a boundary $\Gamma$ of class $C^{\infty}$; as usual,  $\overline{\Omega}:=\Omega\cup\Gamma$. In this case, the space $H^{s,\varphi}_{p}(\Omega)$ admits localization; specifically, it can be described in local coordinates near $\Gamma$ by means of the space $H^{s,\varphi}_{p}(\mathbb{R}^{n}_{+})$, where
\begin{equation*}
\mathbb{R}^{n}_{+}:=\{(x',x_{n}):x'\in\mathbb{R}^{n-1},x_{n}>0\},
\end{equation*}
with
\begin{equation*}
\partial\mathbb{R}^{n}_{+}=\mathbb{R}^{n}_{0}:=
\{(x',x_{n}):x'\in\mathbb{R}^{n-1},x_{n}=0\in\mathbb{R}\}.
\end{equation*}

Consider an arbitrary finite covering of $\overline{\Omega}$ by open bounded sets $U_1,\ldots,U_\varkappa, U_{\varkappa+1},\ldots,U_{\varkappa+r}$ in $\mathbb{R}^{n}$ that satisfy the following two conditions:
\begin{itemize}
\item[(a)] if $1\leq j\leq\varkappa$, then $\Gamma_j:=U_{j}\cap\Gamma\neq\emptyset$ and there exists a $C^\infty$-diffeomorphisms $\pi_{j}:V_{j}\leftrightarrow U_{j}$, where  $V_{j}$ is an open subset of $\mathbb{R}^{n}$ such that $\pi_{j}(V_{j}\cap\mathbb{R}^{n}_+)=U_j\cap\Omega$ and $\pi_{j}(V_{j}\cap\mathbb{R}^{n}_0)=\Gamma_j$;
\item[(b)] if $\varkappa+1\leq j\leq\varkappa+r$, then $\overline{U}_{j}\subset\Omega$.
\end{itemize}
Choose functions $\chi_{1},\ldots,\chi_{\varkappa+r}\in
C^{\infty}(\mathbb{R}^{n})$ such that $\mathrm{supp}\,\chi_{j}\subset U_{j}$ whenever $1\leq j\leq\varkappa+r$ and that $\chi_{1}+\cdots+\chi_{\varkappa+r}=1$ on $\overline{\Omega}$. Let $\mathcal{O}$ and $\mathcal{O_+}$ denote the operators of the extension of a function/distribution by zero onto $\mathbb{R}^{n}$ and $\mathbb{R}^{n}_{+}$, resp.

\begin{theorem}\label{th-localization}
Let $s\in\mathbb{R}$, $\varphi\in\Upsilon$, and $1<p<\infty$. Then the space $H^{s,\varphi}_{p}(\Omega)$ consists of all distributions $u\in \mathcal{S}'(\Omega)$ such that
$\mathcal{O_+}((\chi_{j}u)\circ\pi_{j})\in H^{s,\varphi}_{p}(\mathbb{R}^{n}_+)$ whenever $1\leq j\leq\varkappa$ and that $\mathcal{O}(\chi_{j}u)\in H^{s,\varphi}_{p}(\mathbb{R}^{n})$ whenever $\varkappa+1\leq j\leq\varkappa+r$. Moreover, the following equivalence of norms holds true:
\begin{equation}\label{equiv-norms}
\|u,H^{s,\varphi}_{p}(\Omega)\|\asymp
\sum_{j=1}^{\varkappa}\|\mathcal{O}_{+}((\chi_{j}u)\circ\pi_{j}),
{H^{s,\varphi}_{p}(\mathbb{R}^{n}_+)}\| + \sum_{j=\varkappa+1}^{\varkappa+r}\|\mathcal{O}(\chi_{j}u), {H^{s,\varphi}_{p}(\mathbb{R}^{n})}\|.
\end{equation}
\end{theorem}

Here, of course, $(\chi_{j}u)\circ\pi_{j}$ denotes the image of the distribution $\chi_{j}u$ in the local chart $\pi_{j}$.

\begin{proof}
The flattening mapping
\begin{equation*}
T: u\mapsto\bigl(\mathcal{O}_{+}((\chi_{1}u)\circ\pi_{1}),\ldots,
\mathcal{O}_{+}((\chi_{\varkappa}u)\circ\pi_{\varkappa}),
\mathcal{O}(\chi_{\varkappa+1}u),\ldots,
\mathcal{O}(\chi_{\varkappa+r}u)\bigr),
\end{equation*}
where $u\in\mathcal{S}\,'(\Omega)$, induces a bounded linear operator
\begin{equation}\label{T-operator}
T:H^{s,\varphi}_{p}(\Omega)\to
(H^{s,\varphi}_{p}(\mathbb{R}^{n}_+))^{\varkappa}
\times (H^{s,\varphi}_{p}(\mathbb{R}^{n}))^{r}.
\end{equation}
This is a known fact in the case of Sobolev spaces, when $\varphi(\cdot)\equiv1$ (see, e.g., \cite[Proposition 3.2.3(iii)]{Triebel83}). Hence, the boundedness of this operator in our case follows from Proposition~\ref{prop-quadratic-interp} and Theorem~\ref{th-C-interp-Sobolev-got} by interpolation. (If $s>0$, then the boundedness also follows from \cite[Theorems 2 and~3]{Kalyabin87}.)

Let us construct a left inverse operator $K$ of $T$. For each $j\in\{1,\ldots,\varkappa\}$, we choose a function $\eta_j\in C^{\infty}_{0}(\mathbb{R}^{n})$ such that $\eta_j = 1$ on the set $\pi_{j}^{-1}(\mathrm{supp}\,\chi_j)$ and that $\mathrm{supp}\,\eta_j\subset V_{j}$. Moreover, for each $j\in\{\varkappa +1,\ldots,\varkappa+r\}$,  we choose a function $\eta_j\in C^{\infty}_{0}(\mathbb{R}^{n})$ such that $\eta_j=1$ on $\mathrm{supp}\,\chi_j$ and that $\mathrm{supp}\,\eta_j \subset U_{j}$.
Consider the sewing mapping
\begin{equation*}
K:(v_1,\ldots,v_\varkappa,v_{\varkappa+1},\ldots,v_{\varkappa+r})\mapsto
\sum_{j=1}^{\varkappa}
\mathcal{O}_\Omega((\eta_{j}v_{j})\circ\pi_{j}^{-1})+
\sum_{j=\varkappa+1}^{\varkappa+r}
(\eta_{j}v_{j})\!\!\upharpoonright\!\Omega,
\end{equation*}
where $v_1,\ldots,v_\varkappa\in\mathcal{S}'(\mathbb{R}^{n}_{+})$ and $v_{\varkappa+1},\ldots,v_{\varkappa+r}\in\mathcal{S}'(\mathbb{R}^{n})$.
Here, $\mathcal{O}_\Omega$ denotes the operator of the extension of a distribution by zero onto~$\Omega$. Then $KTu=u$ for every $u\in\mathcal{S}'(\Omega)$. This mapping induces a bounded linear operator
\begin{equation}\label{K-operator}
K: (H^{s,\varphi}_{p}(\mathbb{R}^{n}_+))^{\varkappa}
\times (H^{s,\varphi}_{p}(\mathbb{R}^{n}))^{r}\rightarrow H^{s,\varphi}_{p}(\Omega),
\end{equation}
which is justified by the same reasoning as that used for  \eqref{T-operator}.

The conclusion of Theorem~\ref{th-localization} follows from the boundedness of \eqref{T-operator} and \eqref{K-operator}. Indeed, if $u\in H^{s,\varphi}_{p}(\Omega)$, then $Tu$ lies in the target space of the operator \eqref{T-operator}, i.e. $u$ satisfies the properties indicated in the second sentence of this theorem. Conversely, if a distribution $u\in\mathcal{S}'(\Omega)$ satisfies these properties, then $u=KTu\in H^{s,\varphi}_{p}(\Omega)$. Formula \eqref{equiv-norms} is due to
\begin{equation*}
\|u,H^{s,\varphi}_{p}(\Omega)\|=\|KTu,H^{s,\varphi}_{p}(\Omega)\|\leq \|K\|\,\|Tu\|'\leq \|K\|\,\|T\|\,\|u,H^{s,\varphi}_{p}(\Omega)\|,
\end{equation*}
where  $\|K\|$ and $\|T\|$  are the norms of the operators  \eqref{K-operator} and \eqref{T-operator}, resp., and $\|\cdot\|'$  is the norm in the target space of \eqref{T-operator}.
\end{proof}

Considering $H^{s,\varphi}_{p}(\Omega)$ as a solution space of elliptic boundary-value problems, we need the space of the traces of all distributions $u\in H^{s,\varphi}_{p}(\Omega)$ on $\Gamma$, with $s>1/p$. The trace space is found to be the Besov space of generalized smoothness over $\Gamma$ with smoothness indexes $s-1/p$ and $\varphi$, as in the $\varphi(\cdot)\equiv1$ case. It can be defined on the base of the relevant space over $\mathbb{R}^{n-1}$ with the help of local charts on $\Gamma$. Let us introduce the latter space.

Put $d:=n-1$. Choose a function $\beta_0\in C^{\infty}(\mathbb{R}^{d})$ such that $\beta_0(y)=1$ whenever $|y|\leq1$ and that $\beta_0(y)=0$ whenever $|y|\geq2$, with $y\in\mathbb{R}^{d}$. Assuming  $1\leq k\in\mathbb{Z}$,  we define a function   $\beta_k(y):=\beta_{0}(2^{-k}y)-\beta_{0}(2^{1-k}y)$ of $y\in\mathbb{R}^{d}$. Thus, the system $\{\beta_{k}:0\leq k\in\mathbb{Z}\}$ is a $C^{\infty}$-resolution of unity on $\mathbb{R}^{d}$.

As above, $s\in\mathbb{R}$, $\varphi\in\Upsilon$, and $p\in(1,\infty)$. By definition, the linear space $B^{s,\varphi}_{p}(\mathbb{R}^{d})$ consists of all distributions $w\in\mathcal{S}'(\mathbb{R}^{d})$ such that the distribution $w_{k}:=\mathcal{F}^{-1}[\beta_{k}\mathcal{F}w]$ belongs to $L_{p}(\mathbb{R}^{d})$ whenever $0\leq k\in\mathbb{Z}$ and that
\begin{equation}\label{B-norm}
\|w,B^{s,\varphi}_{p}(\mathbb{R}^{d})\|^{p}:=\sum_{k=0}^{\infty}
\bigl(2^{sk}\varphi(2^{k})\|w_{k},L_{p}(\mathbb{R}^{d})\|\bigr)^{p}
<\infty.
\end{equation}
This space is endowed with the norm $\|w,B^{s,\varphi}_{p}(\mathbb{R}^{d})\|$.

The space $B^{s,\varphi}_{p}(\mathbb{R}^{d})$ belongs to the class of Besov spaces of generalized smoothness investigated, e.g., in \cite{AlmeidaCaetano11, CobosFernandez88, CobosKuhn09, FarkasLeopold06, HaroskeMoura04, LoosveldtNicolay19, Merucci84}. This space does not depend up to equivalence of norms on our choice of the function $\beta_{0}$ \cite[Subsection~2.4]{CobosFernandez88}. If $\varphi(\cdot)\equiv1$, then $B^{s,\varphi}_{p}(\mathbb{R}^{d})$ is the Besov space $B^{s}_{p}(\mathbb{R}^{d})$ (denoted also by $B^{s}_{p,p}(\mathbb{R}^{d})$) with smoothness index $s$ and integrability exponent $p$ (and summability exponent~$p$). Every space $B^{s,\varphi}_{p}(\mathbb{R}^{d})$ is complete and separable, is continuously embedded in $\mathcal{S}'(\mathbb{R}^{d})$, with $\mathcal{S}(\mathbb{R}^{d})$ being continuously and densely embedded in $B^{s,\varphi}_{p}(\mathbb{R}^{d})$. This follows from the analogous properties of the Besov space $B^{0}_{p}(\mathbb{R}^{d})$ because the mapping $w\mapsto w_{s,\varphi}$ (defined by \eqref{w-s-varphi}) sets a topological isomorphism between $B^{s,\varphi}_{p}(\mathbb{R}^{d})$ and $B^{0}_{p}(\mathbb{R}^{d})$ (see, e.g., \cite[Proposition~7]{Merucci84}). Hence, $C^{\infty}_{0}(\mathbb{R}^{d})$ is dense in  $B^{s,\varphi}_{p}(\mathbb{R}^{d})$.

We have the dense continuous embeddings
\begin{equation}\label{ebbeddins-B}
B^{s+\varepsilon}_{p}(\mathbb{R}^{d})\hookrightarrow
B^{s,\varphi}_{p}(\mathbb{R}^{d})\hookrightarrow B^{s-\varepsilon}_{p}(\mathbb{R}^{d})\quad\mbox{for every}\;\;
\varepsilon>0.
\end{equation}
They follow directly from \eqref{B-norm} and the fact that
$c^{-1}t^{-\varepsilon}\leq\varphi(t)\leq c\,t^{\varepsilon}$ whenever $t\geq1$ with some number $c\geq1$, due to property (i) of the class $\Upsilon$ (see \cite[Section 1.5, property~$1^{\circ}$]{Seneta76}). The spaces $B^{s,\varphi}_{p}(\mathbb{R}^{d})$, where $s\in\mathbb{R}$ and
$\varphi\in\Upsilon$, form the refined Besov scale (over $\mathbb{R}^{d}$) with integrability exponent $p$. Note that the Hilbert spaces $B^{s,\varphi}_{2}(\mathbb{R}^{d})$ and $H^{s,\varphi}_{2}(\mathbb{R}^{d})$ coincide up to equivalence of norms \cite[Subsection~4.3, formula~(19)]{Merucci84}.

Let us introduce the space $B^{s,\varphi}_{p}(\Gamma)$. Using the isomorphisms $\pi_{j}:V_j\leftrightarrow U_j$ from condition~(a), we put
$V_{j,0}:=V_j\cap\mathbb{R}^{n}_{0}$ and $\pi_{j,0}:=\pi_{j}\!\!\upharpoonright\!{V_{j,0}}$ whenever $1\leq j\leq\varkappa$. Thus, we have the collection of $\varkappa$ local charts $\pi_{j,0}:V_{j,0}\leftrightarrow \Gamma_{j}$ of class $C^{\infty}$ such that the open sets $\Gamma_{1},\ldots,\Gamma_{\varkappa}$ form a covering of $\Gamma$. We choose functions $\chi_{j,0}\in C^{\infty}(\Gamma)$, with $j\in\{1,\ldots,\varkappa\}$, such that
\begin{equation}\label{1-partition-Gamma}
\mathrm{supp}\,\chi_{j,0}\subset \Gamma_j
\quad\mbox{and}\quad
\chi_{1,0}+\cdots+\chi_{\varkappa,0}\equiv1.
\end{equation}

By definition, the linear space $B^{s,\varphi}_{p}(\Gamma)$ consists of all distributions $v\in \mathcal{D}'(\Gamma)$ such that $\mathcal{O}((\chi_{j,0}\,v)\circ\pi_{j,0})\in B^{s,\varphi}_{p}(\mathbb{R}^{d})$ for every $j\in\{1,\dots,\varkappa\}$. Here, $\mathcal{O}$ denotes the operator of the extension of a function/distribution by zero onto $\mathbb{R}^{d}$. As usual, $\mathcal{D}'(\Gamma)$ stands for the linear topological space of all distributions on the compact boundaryless $C^{\infty}$-manifold~$\Gamma$. The space $B^{s,\varphi}_{p}(\Gamma)$ is endowed with the norm
\begin{equation}\label{Besov-norm-boundary}
\|v,B^{s,\varphi}_{p}(\Gamma)\|:=\sum_{j=1}^{\varkappa}\,
\|\mathcal{O}((\chi_{j,0}\,v)\circ
\pi_{j,0}),B^{s,\varphi}_{p}(\mathbb{R}^{d})\|.
\end{equation}

If $\varphi(\cdot)\equiv1$, then $B^{s,\varphi}_{p}(\Gamma)$ is the Besov space $B^{s}_{p}(\Gamma)$ (see, e.g, \cite[Definition~3.6.1]{Triebel95}). Generally, we have the continuous embeddings \eqref{ebbeddins-B} with $\Gamma$ instead of $\mathbb{R}^{d}$. They are compact due to the fact that $\ell<s$ implies the compact embedding $B^{s}_{p}(\Gamma)\hookrightarrow B^{\ell}_{p}(\Gamma)$ of Besov spaces. Moreover, $B^{s,\varphi}_{p}(\Gamma)$ is continuously embedded in $\mathcal{D}'(\Gamma)$.

\begin{theorem}\label{th-Besov-properties}
Let $s\in\mathbb{R}$, $\varphi\in\Upsilon$, and $1<p<\infty$. Then the space $B^{s,\varphi}_{p}(\Gamma)$ is complete and separable and does not depend on the choice of open sets $\Gamma_{j}$, diffeomorphisms $\pi_{j,0}:V_{j,0}\leftrightarrow \Gamma_{j}$ of class $C^{\infty}$, and functions $\chi_{j,0}\in C^{\infty}(\Gamma)$, with $j\in\{1,\ldots,\varkappa\}$, that satisfy $\Gamma=\Gamma_{1}\cup\cdots\cup\Gamma_{\varkappa}$ and \eqref{1-partition-Gamma}. The set $C^{\infty}(\Gamma)$ is dense in $B^{s,\varphi}_{p}(\Gamma)$.
\end{theorem}

The refined Besov scale has quite similar interpolation properties to those formulated in Theorems \ref{th-C-interp-Sobolev} and \ref{th-C-interp-Sobolev-got}.

\begin{theorem}\label{th-C-interp-Besov}
Under the hypotheses of Theorem~$\ref{th-C-interp-Sobolev}$, the equality
\begin{equation}\label{C-interp-Besov}
[B^{s_0,\varphi_0}_{p_{0}}(W),B^{s_1,\varphi_1}_{p_{1}}(W)]_{\theta}=
B^{s,\varphi}_{p}(W)
\end{equation}
holds true up to equivalence of norms provided that $W=\mathbb{R}^{d}$ or $W=\Gamma$.
\end{theorem}

\begin{theorem}\label{th-C-interp-Besov-got}
Under the hypotheses of Theorem $\ref{th-C-interp-Sobolev-got}$, the equality
\begin{equation*}
B^{s,\varphi}_{p}(W)=[H^{s_0,\varphi_0}_{2}(W),B^{s_1}_{p_1}(W)]_{\theta}
\end{equation*}
holds true up up to equivalence of norms provided that $W=\mathbb{R}^{d}$ or $W=\Gamma$.
\end{theorem}

Theorem \ref{th-C-interp-Besov-got} is a direct consequence of Theorem~\ref{th-C-interp-Besov} and shows (together with Proposition~\ref{prop-quadratic-interp} and its following version for spaces over $\Gamma$) that the refined Besov scale is obtained by interpolation between Sobolev and Besov spaces in two steps.

\begin{proposition}\label{prop-quadratic-interp-Gamma}
Under the hypotheses of Proposition~$\ref{prop-quadratic-interp}$, the equality
\begin{equation*}
H^{s,\varphi}_{2}(\Gamma)=
(H^{s-\varepsilon}_{2}(\Gamma),H^{s+\delta}_{2}(\Gamma))_{\psi}
\end{equation*}
holds true up to equivalence of norms.
\end{proposition}

The proof of this property is given in \cite[Theorem~2.2]{MikhailetsMurach14}.

\begin{proof}[Proof of Theorems $\ref{th-Besov-properties}$ and $\ref{th-C-interp-Besov}$.]
Let us first substantiate \eqref{C-interp-Besov} in the $W=\mathbb{R}^{d}$ case. We need the Banach space
\begin{equation}\label{space-l(E)}
l^{s,\varphi}_{p}(E):=\bigg\{w:=(w_k)_{k=0}^{\infty}\subset E:\|w,l^{s,\varphi}_{p}(E)\|^{p}:=
\sum_{k=0}^{\infty}(2^{sk}\varphi(2^{k})\|w_k,E\|)^{p}<\infty\bigg\},
\end{equation}
endowed with the norm $\|\cdot,l^{s,\varphi}_{p}(E)\|$, where $E:=L_{p}(\mathbb{R}^{d})$. According to \cite[Theorem~2.5(ii)]{CobosFernandez88}, every space $B^{s,\varphi}_{p}(\mathbb{R}^{d})$ is a retract of $l^{s,\varphi}_{p}(L_{p}(\mathbb{R}^{d}))$, where the retraction operator $\mathcal{R}$ and the coretraction operator $\mathcal{T}$ can be chosen independent of $s$, $\varphi$, and~$p$. This means that we have certain bounded linear operators
\begin{equation*}
\mathcal{R}:l^{s,\varphi}_{p}(L_{p}(\mathbb{R}^{d}))\to B^{s,\varphi}_{p}(\mathbb{R}^{d})
\quad\mbox{and}\quad
\mathcal{T}:B^{s,\varphi}_{p}(\mathbb{R}^{d})\to l^{s,\varphi}_{p}(\mathbb{R}^{d})
\end{equation*}
such that $\mathcal{R}\mathcal{T}$ is the identity mapping. Hence,
\begin{equation}\label{C-interp-B-space}
[B^{s_0,\varphi_0}_{p_{0}}(\mathbb{R}^{d}),
B^{s_1,\varphi_1}_{p_{1}}(\mathbb{R}^{d})]_{\theta}
\end{equation}
is a retract of the space
\begin{equation*}
\bigl[l_{p_0}^{s_0,\varphi_0}(L_{p_0}(\mathbb{R}^{d})),
l_{p_1}^{s_1,\varphi_1}(L_{p_1}(\mathbb{R}^{d}))\bigr]_{\theta}=
l^{s,\varphi}_{p}(L_{p}(\mathbb{R}^{d})).
\end{equation*}

Let us justify this equality. The space $l^{s,\varphi}_{p}(E)$ coincides with the Banach space $l_{p}\{E_{k}\}$ of all sequences $w:=(w_k)_{k=0}^{\infty}$ such that each $w_k\in E_{k}$ and
\begin{equation*}
\|w,l_{p}\{E_{k}\}\|:=
\Biggl(\,\sum_{k=0}^{\infty}\|w_k,E_{k}\|^{p}\Biggr)^{1/p}=
\|w,l^{s,\varphi}_{p}(E)\|<\infty,
\end{equation*}
where $E_{k}:=2^{sk}\varphi(2^{k})E$ denotes the space $E$ endowed with the proportional norm $\|\cdot,E_{k}\|:=2^{sk}\varphi(2^{k})\|\cdot,E\|$. Hence,
\begin{align*}
&\bigl[l_{p_0}^{s_0,\varphi_0}(L_{p_0}(\mathbb{R}^{d})),
l_{p_1}^{s_1,\varphi_1}(L_{p_1}(\mathbb{R}^{d}))\bigr]_{\theta}
\nonumber\\
&=\bigl[l_{p_0}\{2^{s_0k}\varphi_0(2^{k})L_{p_0}(\mathbb{R}^{d})\},
l_{p_1}\{2^{s_1k}\varphi_1(2^{k})L_{p_1}(\mathbb{R}^{d})\}\bigr]_{\theta}
\\
&=l_{p}\bigl\{[2^{s_0k}\varphi_0(2^{k})L_{p_0}(\mathbb{R}^{d}),
2^{s_1k}\varphi_1(2^{k})L_{p_1}(\mathbb{R}^{d})]_{\theta}\bigr\}
\\
&=l_{p}\bigl\{2^{sk}\varphi(2^{k})[L_{p_0}(\mathbb{R}^{d}), L_{p_1}(\mathbb{R}^{d})]_{\theta}\bigr\}=
l_{p}\bigl\{2^{sk}\varphi(2^{k})L_{p}(\mathbb{R}^{d})\bigr\}
\\
&=l^{s,\varphi}_{p}(L_{p}(\mathbb{R}^{d}))\nonumber.
\end{align*}
These equalities hold true with equalities of norms. The second equality is due to \cite[Theorems 1.18.1 and Remark 1.18.4/1]{Triebel95}, the third one is valid because the complex interpolation functor is exact of type $\theta$ \cite[Theorem 1.9.3(a)]{Triebel95}, and the fourth one is due to \cite[Theorem 1.18.4 and Remark 1.18.4/2]{Triebel95}.

Thus, the spaces $B^{s,\varphi}_{p}(\mathbb{R}^{d})$ and \eqref{C-interp-B-space} are retracts of the same space  $l^{s,\varphi}_{p}(L_{p}(\mathbb{R}^{d}))$ with the same retraction and coretraction operators. Therefore, equality  \eqref{C-interp-Besov} holds true up to equivalence of norms in the $W=\mathbb{R}^{d}$ case.

This equality for $W=\Gamma$ is deduced from the above case with the help of the flattening mapping
\begin{equation*}
T_0:v\mapsto(\mathcal{O}((\chi_{1,0}\,v)\circ\pi_{1,0}),\ldots,
\mathcal{O}((\chi_{\varkappa,0}\,v)\circ\pi_{\varkappa,0})),
\end{equation*}
defined for all $v\in\mathcal{D}'(\Gamma)$, and the sewing mapping
\begin{equation*}
K_0:(w_{1},\ldots,w_{\varkappa})\mapsto\sum_{j=1}^{\varkappa}\,
\mathcal{O}_{\Gamma}((\eta_{j,0}\,w_{j})\circ\pi_{j,0}^{-1}),
\end{equation*}
defined for all $w_{1},\ldots,w_{\varkappa}\in\mathcal{S}'(\mathbb{R}^{d})$.
Here, $\mathcal{O}_{\Gamma}$ denotes the operator of the extension of a distribution by zero onto~$\Gamma$, and each function $\eta_{j,0}\in C^{\infty}_{0}(\mathbb{R}^{n})$ satisfies the conditions $\mathrm{supp}\,\eta_{j,0}\subset V_{j,0}$ and $\eta_{j,0}=1$ on the set $\pi_{j,0}^{-1}(\mathrm{supp}\,\chi_{j,0})$. We see that $K_{0}T_{0}v=v$ for every $v\in\mathcal{D}'(\Gamma)$ and that the flattening mapping sets an isometric operator
\begin{equation}\label{T0-operetor}
T_0:B^{s,\varphi}_{p}(\Gamma)\to
(B^{s,\varphi}_{p}(\mathbb{R}^{d}))^{\varkappa}
\end{equation}
whenever $s\in\mathbb{R}$, $\varphi\in\Upsilon$, and $1<p<\infty$. Moreover, the sewing mapping sets a bounded operator
\begin{equation}\label{K0-operetor}
K_0:(B^{s,\varphi}_{p}(\mathbb{R}^{d}))^{\varkappa}\rightarrow B^{s,\varphi}_{p}(\Gamma)
\end{equation}
for the same parameters.

Indeed, whatever
\begin{equation*}
w=(w_{1},\ldots,w_{\varkappa})\in
(B^{s,\varphi}_{p}(\mathbb{R}^{d}))^{\varkappa},
\end{equation*}
we have the following:
\begin{align*}
\bigl\|K_0 w, B^{s,\varphi}_{p}(\Gamma)\bigr\| &=\sum_{l=1}^{\varkappa}\;\bigl\|(\chi_{l,0}\,K_0 w)
\circ\pi_{l,0}, B^{s,\varphi}_{p}(\mathbb{R}^{d})\bigr\|\\
&=\sum_{l=1}^{\varkappa}\,\Bigl\|\Bigl(\chi_{l,0}\,\sum_{j=1}^{\varkappa}
\mathcal{O}_{\Gamma}((\eta_{j,0}w_{j})\circ\pi_{j,0}^{-1})\Bigr)
\circ\pi_{l,0},B^{s,\varphi}_{p}(\mathbb{R}^{d})\,\Bigr\|\\
&=\sum_{l=1}^{\varkappa}\;\Bigl\|\,
\sum_{j=1}^{\varkappa}(\zeta_{j,l}\,w_{j})
\circ\lambda_{j,l},B^{s,\varphi}_{p}(\mathbb{R}^{d})\,\Bigr\| \\
&\leq c\sum_{j=1}^{\varkappa}
\|w_{j},B^{s,\varphi}_{p}(\mathbb{R}^{d})\|.
\end{align*}
Here, $\zeta_{j,l}:=\mathcal{O}(\chi_{l,0}\circ\pi_{j,0})\,\eta_{j,0}\in
C_{0}^{\infty}(\mathbb{R}^{d})$, and
$\lambda_{j,l}:\mathbb{R}^{d}\,\leftrightarrow\,\mathbb{R}^{d}$ is a $C^{\infty}$-diffeomorphism such that $\lambda_{j,l}=\pi_{j,0}^{-1}\circ\pi_{l,0}$ in a neighbourhood of the set $\mathrm{supp}\,\zeta_{j,l}$ and that $\lambda_{j,l}(x)=x$ for all points $x\in\mathbb{R}^{d}$ with $\|x\|\gg1$. The number $c>0$ does not depend on $w$ because the operator of the multiplication by a function of class $C_{0}^{\infty}(\mathbb{R}^{d})$ and the operator of the change of variables given by $\lambda_{j,l}$ are bounded on $B^{s,\varphi}_{p}(\mathbb{R}^{d})$. This is known for the classical Besov spaces (see, e.g., \cite[Theorems 2.8.2(i) and 2.10.2(i)]{Triebel83}) and then follows for every $B^{s,\varphi}_{p}(\mathbb{R}^{d})$
by the interpolation in view of Proposition~\ref{prop-quadratic-interp} and Theorem~\ref{th-C-interp-Besov-got} proved above in the $W=\mathbb{R}^{d}$ case.

Thus, each space $B^{s,\varphi}_{p}(\Gamma)$ is a retract of  $(B^{s,\varphi}_{p}(\mathbb{R}^{d}))^{\varkappa}$ with the retraction operator $K_{0}$ and coretraction operator $T_{0}$. Hence, $B^{s,\varphi}_{p}(\Gamma)$ is complete because the space  $(B^{s,\varphi}_{p}(\mathbb{R}^{d}))^{\varkappa}$ is complete. Moreover, if a set $M$ is dense in the latter space, then $K_{0}(M)$ is dense in $B^{s,\varphi}_{p}(\Gamma)$. Hence, $B^{s,\varphi}_{p}(\Gamma)$ is separable, and the set $K_{0}((C^{\infty}_{0}(\mathbb{R}^{d}))^{\varkappa})\subset C^{\infty}(\Gamma)$ is dense in $B^{s,\varphi}_{p}(\Gamma)$, in view of the relevant properties of $B^{s,\varphi}_{p}(\mathbb{R}^{d})$.

It follows from the boundedness of operators \eqref{T0-operetor} and
\eqref{K0-operetor} and from the interpolation formula \eqref{C-interp-Besov} proved for $W=\mathbb{R}^{d}$ that the flattening and sewing operators are bounded on the pairs of spaces
\begin{equation*}
T_0:[B^{s_0,\varphi_0}_{p_{0}}(\Gamma),
B^{s_1,\varphi_1}_{p_{1}}(\Gamma)]_{\theta}\to
(B^{s,\varphi}_{p}(\mathbb{R}^{d}))^{\varkappa}
\end{equation*}
and
\begin{equation*}
K_0:(B^{s,\varphi}_{p}(\mathbb{R}^{d}))^{\varkappa}\to
[B^{s_0,\varphi_0}_{p_{0}}(\Gamma),
B^{s_1,\varphi_1}_{p_{1}}(\Gamma)]_{\theta}.
\end{equation*}
Thus, the spaces $B^{s,\varphi}_{p}(\Gamma)$ and $[B^{s_0,\varphi_0}_{p_{0}}(\Gamma),
B^{s_1,\varphi_1}_{p_{1}}(\Gamma)]_{\theta}$ are retracts of the same space $(B^{s,\varphi}_{p}(\mathbb{R}^{d}))^{\varkappa}$ with the same retraction and coretraction operators. Therefore, equality \eqref{C-interp-Besov} holds true up to equivalence of norms in the $W=\Gamma$ case.

Now the independence of the space $B^{s,\varphi}_{p}(\Gamma)$ of $\{\Gamma_{j}\}$, $\{\pi_{j,0}\}$, and $\{\chi_{j,0}\}$ follows by \eqref{C-interp-Besov} from the known fact that this independence holds true for the classical Besov spaces (see, e.g., \cite[Proposition 3.2.3(ii)]{Triebel83}).
\end{proof}

\begin{remark}
Applications of the complex interpolation to Besov spaces of generalized smoothness over $\mathbb{R}^{d}$ were given by Triebel \cite[Theorem~4.2/2]{Triebel77III}, Besov \cite[Theorem~2]{Besov05}, Knopova \cite[Lemma~1]{Knopova06}, Baghdasaryan \cite[Theorem 2.4(i)]{Baghdasaryan10}, Drihem \cite[Theorem~4.24]{Drihem23CMUC}. These results do not cover Theorem~\ref{th-C-interp-Besov} in the $W=\mathbb{R}^{d}$ case. The Besov spaces of generalized smoothness used by Triebel do not coincide with the classical Besov spaces for corresponding (power) smoothness indexes \cite[Theorem~6.2/1]{Triebel77III}. The result by Besov concerns spaces of locally integrable functions (the $s>0$ case in relation to $B^{s,\varphi}_{p}(\mathbb{R}^{d})$). Knopova and Baghdasaryan interpolate between distribution spaces whose generalized (nonclassical) smoothness is characterized in a certain way by the same non-numeric parameter. The result by Drihem concerns homogeneous distribution spaces. However, it is noted \cite[Remark~4.29]{Drihem23CMUC} that this result can be extended to inhomogeneous spaces, specifically to Besov spaces of generalized smoothness \cite[Example~4.6]{Drihem23AOT}. Note that Besov \cite{Besov05} and Drihem \cite{Drihem23CMUC} studied wide classes of distribution spaces whose variable generalized smoothness is given by a sequence of weight functions depending on spatial variables.
\end{remark}

\begin{theorem}\label{th-trace}
Let $1<p<\infty$, $s>1/p$, and $\varphi\in\Upsilon$. Then the trace mapping $R_{\Gamma}:u\mapsto u\!\upharpoonright\!\Gamma$, where $u\in C^{\infty}(\overline{\Omega})$, extends uniquely (by continuity) to a bounded linear operator
\begin{equation}\label{trace-operator}
R_{\Gamma}:H^{s,\varphi}_{p}(\Omega)\to B^{s-1/p,\varphi}_{p}(\Gamma).
\end{equation}
This operator is onto and has a bounded linear right inverse operator
\begin{equation}\label{trace-operator-reverse}
S_{\Gamma}: B^{s-1/p,\varphi}_{p}(\Gamma)\rightarrow H^{s,\varphi}_{p}(\Omega)
\end{equation}
such that $S_{\Gamma}$ does not depend on $p$, $s$, and $\varphi$ as a mapping. Hence, the equivalence of norms
\begin{equation}\label{trace-norms-equivalence}
\|v,B^{s-1/p,\varphi}_{p}(\Gamma)\|\asymp
\inf\bigl\{\|u,H^{s,\varphi}_{p}(\Omega)\|:u\in H^{s,\varphi}_{p}(\Omega), v=R_{\Gamma}u\bigr\},
\end{equation}
holds true with respect to $v\in B^{s-1/p,\varphi}_{p}(\Gamma)$.
\end{theorem}

As usual, $u\!\upharpoonright\!\Gamma$ denotes the restriction of the function $u$ to $\Gamma$, and $C^{\infty}(\overline{\Omega})$ stands for  the set of the restrictions of all functions $w\in C^{\infty}(\mathbb{R}^{n})$ to $\overline{\Omega}$. Since $\Omega$ is bounded, $C^{\infty}(\overline{\Omega})=C^{\infty}_{0}(\overline{\Omega})$ is a dense subset of $H^{s,\varphi}_{p}(\Omega)$.

\begin{proof}
This result is known for $\varphi(\cdot)\equiv1$ (e.g., it is contained in \cite[Lemma~5.4.4]{Triebel95}) and for $p=2$ \cite[Theorem~3.5]{MikhailetsMurach14}. We will deduce it from this fact by the interpolation Theorems \ref{th-C-interp-Sobolev-got} and~\ref{th-C-interp-Besov-got}. We need to consider the cases $1<p<2$ and $2<p<\infty$ separately, which is caused by the restriction $s>1/p$.

We first consider the case where $1<p<2$. Choose real numbers $p_{1}>1$ and $s_{1}$ so that $1/p<1/p_{1}<s_{1}<s$, and note that $1<p_{1}<p<2$. We then define the parameters $\theta$, $s_{0}$, and $\varphi_{0}$ according to Theorem~\ref{th-C-interp-Sobolev-got} and see that $s_{0}>s>1/p>1/2$ in the case under consideration. Therefore, we have the following two pairs of bounded linear operators:
\begin{gather}\label{pair-trace-operators}
R_{\Gamma}:H^{s_0,\varphi_0}_{2}(\Omega)\rightarrow H^{s_{0}-1/2,\varphi_0}_{2}(\Gamma)\quad\mbox{and}\quad
R_{\Gamma}:H^{s_1}_{p_1}(\Omega)\rightarrow B^{s_1-1/p_1}_{p_1}(\Gamma),\\
S_{\Gamma}:H^{s_{0}-1/2,\varphi_0}_{2}(\Gamma)\rightarrow H^{s_0,\varphi_0}_{2}(\Omega)\quad\mbox{and}\quad
S_{\Gamma}:B^{s_1-1/p_1}_{p_1}(\Gamma)\rightarrow H^{s_1}_{p_1}(\Omega). \label{pair-trace-operators-reverse}
\end{gather}
According to Theorems \ref{th-C-interp-Sobolev-got}  and    \ref{th-C-interp-Besov-got}, the equalities \eqref{C-interp-Sobolev-got}  and
\begin{equation*}
B^{s-1/p,\varphi}_{p}(\Gamma)=
[H^{s_{0}-1/2,\varphi_0}_{2}(\Gamma),B^{s_1-1/p_1}_{p_1}(\Gamma)]_{\theta}
\end{equation*}
hold true up to equivalence of norms. Applying the complex interpolation with the parameter $\theta$ to these pairs of operators, we obtain  bounded operators \eqref{trace-operator} and \eqref{trace-operator-reverse} such that $\eqref{trace-operator}$ is an extension by continuity of the mapping $C^{\infty}(\overline{\Omega})\ni u\mapsto u\!\upharpoonright\!\Gamma$ and that $\eqref{trace-operator-reverse}$ is a right-inverse of \eqref{trace-operator}. Hence, the operator \eqref{trace-operator} is onto, and \eqref{trace-norms-equivalence} holds true.

Let us now consider the second case where $p>2$. We choose a number $p_{1}>p$ (e.g., put $p_{1}:=2p$), define $\theta\in(0,1)$ by formula \eqref{p-2-interp}, choose a number $s_{1}>1/p_{1}$ so that $s-1/p>\theta(s_1-1/p_{1})$, and define a number $s_{0}$ by the formula $s=(1-\theta)s_{0}+\theta s_{1}$. Then
\begin{equation*}
(1-\theta)\Bigl(s_{0}-\frac{1}{2}\Bigr)=s-\theta s_1-\frac{1-\theta}{2}=
s-\theta s_1-\frac{1}{p}+\frac{\theta}{p_1}>0,
\end{equation*}
i.e. $s_{0}>1/2$. We finally define $\varphi_{0}$ according to Theorem~\ref{th-C-interp-Sobolev-got} and have the pairs \eqref{pair-trace-operators} and \eqref{pair-trace-operators-reverse} of bounded linear operators. Since the operators in each pair coincide on the intersection of their domains, they naturally induce linear operators on the algebraic sums of their domains and, hence, can be interpolated. (In the first case, the second operator of each pair is an extension of the first one because $s_{1}<s_{0}$ and $p_{1}<2$.) We now complete the proof in the same manner as in the first case.
\end{proof}

\begin{remark}\label{rem-trace-hyperplane}
A version of Theorem~\ref{th-trace} is valid for the trace mapping $R_{0}:u\mapsto u\!\upharpoonright\!\mathbb{R}^{n}_{0}$, where $u\in C^{\infty}_{0}(\mathbb{R}^{n})$ or $u\in C^{\infty}_{0}(\overline{\mathbb{R}^{n}_{+}})$, with an analogous proof.
A right inverse $S_{0}$ of $R_{0}$ can be constructed with the help of
the isomorphism
\begin{equation}\label{Dirichlet-isomorphism}
(1-\Delta,R_{0}):H^{s}_{p}(\mathbb{R}^{n}_{+})\leftrightarrow
H^{s-2}_{p}(\mathbb{R}^{n}_{+})\times B^{s-1/p}_{p}(\mathbb{R}^{n-1}),
\end{equation}
where $s>1/p$. This isomorphism is induced by the Dirichlet problem for the elliptic equation $u-\Delta u=f$ in the half-space $\mathbb{R}^{n}_{+}$ (as usual, $\Delta$ is the Laplace operator), we naturally identifying $\mathbb{R}^{n}_{0}$ with $\mathbb{R}^{n-1}$. The operator $S_{0}$ can be defined as the restriction of the inverse of \eqref{Dirichlet-isomorphism} to the space $\{0\}\times B^{s-1/p}_{p}(\mathbb{R}^{n-1})$. Thus, $B^{s-1/p,\varphi}_{p}(\mathbb{R}^{n-1})$ is the trace space for $H^{s,\varphi}_{p}(\mathbb{R}^{n})$ and $H^{s,\varphi}_{p}(\mathbb{R}^{n}_{+})$ whenever $1<p<\infty$, $s>1/p$, and $\varphi\in\Upsilon$.
\end{remark}

\begin{remark}
A necessary and sufficient condition for the Sobolev space of generalized smoothness over $\mathbb{R}^{n}$ under which every function from the space has a trace on a hyperplane was obtained by H\"ormander \cite[Theorem 2.2.8]{Hermander63} in the case where the integrability exponent $p$ equals $2$ (see also Volevich and Paneah's paper \cite[Theorems 6.1 and~6.2]{VolevichPaneah65}) and by Goldman \cite[Theorem~1]{Goldman76} for each $p\in(1,\infty)$. The result by H\"ormander concerns some (generally, anisotropic) spaces of generalized smoothness with the integrability exponent $p\in[1,\infty]$ and also describes their trace spaces. These spaces coincide with Sobolev spaces of generalized smoothness if and only if $p=2$. Goldman \cite[Theorem~2]{Goldman76} indicated a Besov space of generalized smoothness that contains these traces but did not assert that it is the narrowest space. The trace spaces were found by Kalyabin \cite[\S~1]{Kalyabin78} for certain anisotropic Sobolev spaces of generalized smoothness; however, the class of these spaces does not contain isotropic ones. For isotropic Sobolev spaces of generalized smoothness, the description of the trace spaces contains in Kalyabin's later result \cite[Section~1, Main Theorem]{Kalyabin81} devoted to Triebel--Lizorkin spaces of generalized smoothness (both isotropic and anisotropic). Goldman \cite[Corollary~3]{Goldman76} and Merucci \cite[Theorem~15]{Merucci84} indicated classes of Sobolev spaces of generalized smoothness for which the traces belong to the Besov space whose smoothness index equals $t^{-1/p}\omega(t)$ provided that the function $\omega$ is the smoothness index of the Sobolev space. However, it was not shown that this Besov space is the trace space (see, e.g., \cite[Remark~5]{Merucci84}). The spaces $H^{s,\varphi}_{p}(\mathbb{R}^{n})$ for which $s>1/p$ and $\varphi\in\Upsilon$ belong to these classes, with $\omega(t)\equiv t^{s}\varphi(t)$.
\end{remark}

\begin{remark}
The hypothesis $s>1/p$ is essential in Theorem~\ref{th-trace}. Indeed, since the set $C^{\infty}_{0}(\Omega)$ is dense in the Sobolev space $H^{1/p}_{p}(\Omega)$ \cite[Theorem 4.3.2/1(a)]{Triebel95}, the trace operator $R_{\Gamma}$ cannot be well defined as a continuous mapping from whole $H^{1/p}_{p}(\Omega)$ to $\mathcal{D}'(\Gamma)$, for otherwise we would get by closure that $R_{\Gamma}u=0$ for every $u\in H^{1/p}_{p}(\Omega)$, e.g, for $u(\cdot)=1$ on $\overline{\Omega}$. Here, $C^{\infty}_{0}(\Omega)$ stands for the set of all functions $u\in C^{\infty}(\overline{\Omega})$ such that $\operatorname{supp}u\subset\Omega$. Hence, the continuous trace mapping $R_{\Gamma}:H^{s,\varphi}_{p}(\Omega)\to\mathcal{D}'(\Gamma)$ is not well defined whatever $s<1/p$ and $\varphi\in\Upsilon$, in view of the continuous embedding $H^{1/p}_{p}(\Omega)\hookrightarrow H^{s,\varphi}_{p}(\Omega)$. As to the limiting case $s=1/p$, the continuous trace mapping $R_{\Gamma}:H^{1/p,\varphi}_{p}(\Omega)\to\mathcal{D}'(\Gamma)$ is well defined if and only if the function parameter $\varphi\in\Upsilon$ satisfies
\begin{equation}\label{integral-condition}
\int\limits_{1}^{\infty}\frac{dt}{t\,\varphi^{p'}(t)}<\infty,
\end{equation}
which follows from \cite[Theorem~1]{Goldman76}; as usual, $1/p+1/p'=1$.
If \eqref{integral-condition} holds true, then the trace space for $H^{1/p,\varphi}_{p}(\Omega)$ differs from $B^{0,\varphi}_{p}(\Gamma)$. This follows, e.g., from \cite[Theorem 2.2.8]{Hermander63} or \cite[Theorem~6.2]{VolevichPaneah65} in the $p=2$ case, as is shown in \cite[Theorem 3.8~b)]{MikhailetsMurach06UMJ3}. The same concerns the trace mapping $R_{0}$ indicated in Remark~\ref{rem-trace-hyperplane}. Running ahead, note that the Fredholm property of elliptic boundary-value problems in $\Omega$ will not hold true if we suppose \eqref{integral-condition} and take $H^{1/p,\varphi}_{p}(\Omega)$ as a solution space of these problems. This is explained in \cite[Section~9.3]{MikhailetsMurach12BJMA2} in the $p=2$ case.
\end{remark}

\section{Applications to elliptic problems}\label{sec3}

Recall that $\Omega$ is a bounded open Euclidean domain of dimension $n\geq2$ with a boundary $\Gamma$ of class $C^{\infty}$. We choose integers $q\geq1$ and $m_{1},\ldots,m_{q}\geq0$ arbitrarily and consider an elliptic boundary-value problem of the form
\begin{gather}\label{ell-equ}
Au=f\quad\mbox{in}\quad\Omega,\\
B_{j}u=g_{j}\quad\mbox{on}\quad\Gamma,\quad j=1,...,q. \label{bound-cond}
\end{gather}
Here, $A$ is a linear partial differential operator (PDO) on $\overline{\Omega}$ of order $2q$, whereas each $B_{j}$ is a linear boundary PDO on $\Gamma$ of order $m_{j}$. We suppose that all coefficients of the PDOs $A$ and $B_{j}$ belong to $C^{\infty}(\overline{\Omega})$ and $C^{\infty}(\Gamma)$, resp. Put $m:=\max\{m_{1},\ldots,m_{q}\}$, $B:=(B_{1},\ldots,B_{q})$, and $g:=(g_{1},\ldots,g_{q})$. We admit the case where $m\geq2q$.

Recall that the ellipticity of this problem means that the PDO $A$ is properly elliptic on $\overline{\Omega}$ and that the collection $B$ of boundary PDOs satisfies the Lopatinskii condition with respect to $A$ on $\Gamma$ (see, e.g., \cite[Section~1.2]{Agranovich97}). (If $n\geq3$, then the elipticity of $A$ implies its proper ellipticity.)

\begin{theorem}\label{th-Fredholm}
Let $p\in(1,\infty)$, $s>m+1/p$, and $\varphi\in\Upsilon$. Then the mapping
\begin{equation}\label{mapping-(A,B)}
u\mapsto(Au,B_{1}u,\ldots,B_{q}u)=(Au,Bu),\quad\mbox{where}\;\;
u\in C^{\infty}(\overline{\Omega}),
\end{equation}
extends uniquely (by continuity) to a bounded linear operator
\begin{equation}\label{(A,B)}
(A,B):H^{s,\varphi}_{p}(\Omega)\to H^{s-2q,\varphi}_{p}(\Omega)\times
\prod_{j=1}^{q}B^{s-m_{j}-1/p,\varphi}_{p}(\Gamma)=:
\Xi^{s-2q,\varphi}_{p} (\Omega,\Gamma).
\end{equation}
This operator is Fredholm. Its kernel lies in $C^{\infty}(\overline{\Omega})$ and together with its index does not depend on the parameters $s$, $\varphi$, and~$p$. The target space of this operator splits into the direct sum
\begin{equation}\label{M+range}
\Xi^{s-2q,\varphi}_{p} (\Omega,\Gamma)=M\dotplus(A,B)(H^{s,\varphi}_{p}(\Omega))
\end{equation}
for a certain finite-dimensional space $M\subset C^{\infty}(\overline{\Omega})\times(C^{\infty}(\Gamma))^{l}$ that does not depend on these parameters.
\end{theorem}

If $\varphi(\cdot)\equiv1$, we will omit $\varphi$ in the designation $\Xi^{s-2q,\varphi}_{p} (\Omega,\Gamma)$ of the target space of operator~\eqref{(A,B)}.

Recall that a bounded linear operator $T:X\to Y$ acting between Banach spaces $X$ and $Y$ is called a Fredholm one if its kernel $\ker T$ and co-kernel $\operatorname{coker}T:=Y/T(X)$ are finite-dimensional. If $T$ is a Fredholm operator, then its range $T(X)$ is closed in $Y$  \cite[Lemma~19.1.1]{Hermander07v3} and its index $\mathrm{ind}\,T:=\dim\ker T-\dim\operatorname{coker}T$ is well defined and finite.

We prove Theorem~\ref{th-Fredholm} by means of the interpolation of Fredholm operators and rely on the following result:

\begin{proposition}\label{prop-interpolation-Fredholm}
Let $X=[X_{\,0},X_{1}]$ and  $Y=[Y_{\,0},Y_{1}]$ be ordered pairs of
Banach spaces such that the dense continuous embeddings $X_1\hookrightarrow X_0$ and $Y_1\hookrightarrow Y_0$ (or $X_0\hookrightarrow X_1 $ and $Y_0\hookrightarrow Y_1$) hold true. Suppose that a linear mapping $T$ is given on $X_{\,0}$  (resp. $X_{\,1}$) and defines bounded Fredholm operators $T: X_{0}\rightarrow Y_{0}$ and $T:X_{1}\rightarrow Y_{1}$ with common kernel
$N$ and the same index $\varkappa$. Then for any interpolation parameter $\theta\in(0,1)$ the
bounded operator $T: [X_{\,0},X_{1}]_{\theta} \rightarrow [Y_{\,0},Y_{1}]_{\theta}$ is a Fredholm operator with the kernel $N$, range $[Y_{\,0},Y_{1}]_{\theta}\cap T(X_{\,0})$  (resp. $[Y_{\,0},Y_{1}]_{\theta}\cap T(X_{\,1})$) and the same index $\varkappa$.
\end{proposition}

This result can be considered as a consequence of \cite[Proposition~5.2]{Geymonat65} and as a version of \cite[Theorem~1.7]{MikhailetsMurach14} for the complex interpolation with an analogous proof.

\begin{proof}[Proof of Theorem~$\ref{th-Fredholm}$]
If $\varphi(\cdot)\equiv1$, then this theorem is a classical result
(see, e.g., \cite[Theorem~5.5.2]{Triebel95} in the case where the collection $(A,B_{1},\ldots,B_{q})$ is regularly elliptic \cite[Definition~5.2.1/4]{Triebel95} and $s\geq2q$; this result is also contained in \cite[Theorem~5.2]{Johnsen96}). If $p=2$ and $\varphi\in\Upsilon$ (the case of Hilbert spaces), Theorem~\ref{th-Fredholm} is contained in \cite[Theorem~4.1]{MikhailetsMurach06UMJ3} (if $m\leq2q-1$) and \cite[Theorem~1]{KasirenkoMurach17UMJ11} (if $m\geq2q$). Let $N$ and $\varkappa$ respectively denote the kernel and index of the Fredholm operator \eqref{(A,B)} provided that $1<p<\infty$ and $\varphi(\cdot)\equiv1$ or that $p=2$ and $\varphi\in\Upsilon$, with $N$ and $\varkappa$ not depending on $\varphi$ and $p$ indicated. Let $M$ satisfy the conclusion of Theorem~$\ref{th-Fredholm}$ for these $\varphi$ and $p$. We need to treat the cases where $1<p<2$ and where $p>2$ separately.

We first consider the $1<p<2$ case. Choose numbers $p_1$ and $s_1$ so that $1<p_1<p$ and $m+1/p<m+1/p_1<s_1<s$. Define parameters $\theta\in(0,1)$,  $s_0$, and $\varphi_{0}\in\Upsilon$ according to the hypotheses of Theorem~\ref{th-C-interp-Sobolev-got}. Note that $s_0>s_1>m+1/p_1>m+1/2$ and $2>p>p_1$. Hence, we have the dense continuous  embeddings
\begin{equation}\label{(A,B)-embeddings}
H^{s_0,\varphi_0}_2(\Omega)\hookrightarrow  H^{s_1}_{p_1}(\Omega)
\quad\mbox{and}\quad
\Xi^{s_0-2q,\varphi_0}_2 (\Omega,\Gamma)\hookrightarrow
\Xi^{s_1-2q}_{p_1}(\Omega,\Gamma).
\end{equation}
Moreover, we have the bounded Fredholm operators
\begin{equation}\label{Operator1}
(A,B):H^{s_0,\varphi_0}_{2}(\Omega)\to
\Xi^{s_0-2q,\varphi_0}_{2}(\Omega,\Gamma)
\end{equation}
and
\begin{equation}\label{Operator2}
(A,B):H^{s_1}_{p_1}(\Omega)\to \Xi^{s_1-2q,\varphi_1}_{p_1}(\Omega,\Gamma)
\end{equation}
with the kernel $N$ and index $\varkappa$. Interpolating them, we conclude by Theorems \ref{th-C-interp-Sobolev-got} and \ref{th-C-interp-Besov-got} that the restriction of the latter operator to the space
\begin{equation}\label{source-space-interpolation}
[H^{s_0,\varphi_0}_{2}(\Omega),H^{s_1}_{p_1}(\Omega)]_{\theta}=
H^{s,\varphi}_{p}(\Omega)
\end{equation}
is a bounded operator between the spaces $H^{s,\varphi}_{p}(\Omega)$ and
\begin{align}
&[\Xi^{s_0-2q,\varphi_0}_{2}(\Omega,\Gamma),
\Xi^{s_1-2q}_{p_1}(\Omega,\Gamma)]_\theta \notag\\
&=[H^{s_0-2q,\varphi_0}_{2}(\Omega),
H^{s_1-2q}_{p_1}(\Omega)]_{\theta}\times
\prod_{j=1}^{q}[H^{s_0-m_{j}-1/2,\varphi_0}_{2}(\Gamma),
B^{s_1-m_{j}-1/p_1}_{p_1}(\Gamma)]_\theta  \notag\\
&=H^{s-2q,\varphi}_{p}(\Omega)\times
\prod_{j=1}^{q}B^{s-m_{j}-1/p,\varphi}_{p}(\Gamma)=
\Xi^{s-2q,\varphi}_{p}(\Omega,\Gamma). \label{target-space-interplation}
\end{align}
Owing to Proposition~\ref{prop-interpolation-Fredholm} and embeddings \eqref{(A,B)-embeddings}, this operator is Fredholm with the kernel $N$, index $\varkappa$, and range
\begin{equation*}
\Xi^{s-2q,\varphi}_{p}(\Omega,\Gamma)\cap(A,B)(H^{s_1}_{p_1}(\Omega)).
\end{equation*}
Moreover,
\begin{align*}
\Xi^{s-2q,\varphi}_{p}(\Omega,\Gamma)&=
\Xi^{s-2q,\varphi}_{p}(\Omega,\Gamma)\cap
\Xi^{s_1-2q}_{p_1}(\Omega,\Gamma)=
\Xi^{s-2q,\varphi}_{p}(\Omega,\Gamma)\cap
\bigl(M\dotplus(A,B)(H^{s_1}_{p_1}(\Omega))\bigr)\\
&=M\dotplus\bigl(\Xi^{s-2q,\varphi}_{p}(\Omega,\Gamma)\cap
(A,B)(H^{s_1}_{p_1}(\Omega))\bigr)=
M\dotplus(A,B)(H^{s,\varphi}_{p}(\Omega))
\end{align*}
because $M\subset\Xi^{s-2q,\varphi}_{p}(\Omega,\Gamma)$; i.e. \eqref{M+range} holds true. Since $C^{\infty}(\overline{\Omega})$ is dense in $H^{s,\varphi}_{p}(\Omega)$, this operator is the extension by continuity of mapping \eqref{mapping-(A,B)}. We have proved Theorem~$\ref{th-Fredholm}$ in the case where $1<p<2$.

Let us now consider the $p>2$ case. We separately treat the cases where $s>m+1/2$ and where $m+1/p<s\leq m+1/2$. Assume first that $s>m+1/2$. Choose numbers $p_1$ and $s_0$ so that $p_1>p$ and $m+1/2<s_0<s$. Define parameters $\theta\in(0,1)$,  $s_1$, and $\varphi_{0}\in\Upsilon$ according to the hypotheses of Theorem~\ref{th-C-interp-Sobolev-got}; then $s_1>s_0>m+1/2>m+1/p_1$. We have the dense continuous  embeddings
\begin{equation}\label{(A,B)-embeddings-reverse}
H^{s_1}_{p_1}(\Omega)\hookrightarrow H^{s_0,\varphi_0}_2(\Omega)
\quad\mbox{and}\quad
\Xi^{s_1-2q}_{p_1}(\Omega,\Gamma) \hookrightarrow\Xi^{s_0-2q,\varphi_0}_{2}(\Omega,\Gamma),
\end{equation}
and, moreover, the Fredholm bounded operators \eqref{Operator1} and \eqref{Operator2} with the kernel $N$ and index~$\varkappa$. The interpolation formulas \eqref{source-space-interpolation} and \eqref{target-space-interplation} hold true due to Theorems \ref{th-C-interp-Sobolev-got} and \ref{th-C-interp-Besov-got}. We conclude by Proposition~\ref{prop-interpolation-Fredholm} and embeddings \eqref{(A,B)-embeddings-reverse} that the restriction of \eqref{Operator1} to space \eqref{target-space-interplation}
is a bounded Fredholm operator on the pair of spaces \eqref{source-space-interpolation} and \eqref{target-space-interplation} and that this operator has the kernel $N$, index $\varkappa$, and range
\begin{equation*}
\Xi^{s-2q,\varphi}_{p}(\Omega,\Gamma)\cap
(A,B)(H^{s_0,\varphi_0}_{2}(\Omega)).
\end{equation*}
Hence,
\begin{align*}
\Xi^{s-2q,\varphi}_{p}(\Omega,\Gamma)&=
\Xi^{s-2q,\varphi}_{p}(\Omega,\Gamma)\cap
\bigl(M\dotplus(A,B)(H^{s_0,\varphi_0}_{2}(\Omega))\bigr)\\
&=M\dotplus\bigl(\Xi^{s-2q,\varphi}_{p}(\Omega,\Gamma)\cap
(A,B)(H^{s_0,\varphi_0}_{2}(\Omega))\bigr)=
M\dotplus(A,B)(H^{s,\varphi}_{p}(\Omega)),
\end{align*}
i.e. \eqref{M+range} holds true. We have proved Theorem~$\ref{th-Fredholm}$ in the case where $p>2$ and $s>m+1/2$.

Consider the case where $p>2$ and $m+1/p<s\leq m+1/2$. Choose numbers $s_0$ and $s_1$ so that
\begin{equation*}
m+1/p<s_0<s\leq m+1/2<s_1;
\end{equation*}
define a number $\theta\in(0,1)$ according to the formula  $s=(1-\theta)s_0 +\theta s_1$, and put $\varphi_1(t):=\varphi^{1/\theta}(t)$ for every $t\geq1$, with $\varphi_1\in\Upsilon$. We have the  dense continuous embeddings
\begin{equation}\label{embeddings-fixed-p}
H^{s_1,\varphi_1}_{p}(\Omega)\hookrightarrow H^{s_0}_{p}(\Omega)
\quad\mbox{and}\quad
\Xi^{s_1-2q,\varphi_1}_{p}(\Omega,\Gamma)\hookrightarrow
\Xi^{s_0-2q}_{p}(\Omega,\Gamma).
\end{equation}
We also have the bounded Fredholm operators
\begin{equation*}
(A,B):H^{s_0}_{p}(\Omega)\to\Xi^{s_0-2q}_{p}(\Omega,\Gamma)
\end{equation*}
and
\begin{equation*}
(A,B):H^{s_1,\varphi_1}_{p}(\Omega)\to
\Xi^{s_1-2q,\varphi_1}_{p}(\Omega,\Gamma)
\end{equation*}
with the kernel $N$ and index $\varkappa$. The Fredholm property of the second operator was proved in the previous paragraph because $s_1>m+1/2$. Interpolating these operators, we conclude by Theorems  \ref{th-C-interp-Sobolev} and \ref{th-C-interp-Besov} that the restriction of the first operator to the space
\begin{equation*}
[H^{s_0}_{p}(\Omega),H^{s_1,\varphi_1}_{p}(\Omega)]_{\theta}=
H^{s,\varphi}_{p}(\Omega)
\end{equation*}
acts continuously between $H^{s,\varphi}_{p}(\Omega)$ and
\begin{equation*}
[\Xi^{s_0-2q}_{p}(\Omega,\Gamma),
\Xi^{s_1-2q,\varphi_1}_{p}(\Omega,\Gamma)]_{\theta}=
\Xi^{s-2q,\varphi}_{p}(\Omega,\Gamma).
\end{equation*}
Owing to Proposition~\ref{prop-interpolation-Fredholm} and embeddings \eqref{embeddings-fixed-p}, this operator is Fredholm with the kernel $N$, index $\varkappa$, and range
\begin{equation*}
(A,B)(H^{s,\varphi}_{p}(\Omega))=
\Xi^{s-2q,\varphi}_{p}(\Omega,\Gamma)\cap(A,B)(H^{s_0}_{p}(\Omega)).
\end{equation*}
This equality yields \eqref{M+range} as above. We have proved Theorem~$\ref{th-Fredholm}$ in the last case.
\end{proof}

Let $\Lambda_{m}(\Omega)$ denote the union of all spaces $H^{s}_{p}(\Omega)$ such that $p\in(1,\infty)$ and $s>m+1/p$. (The set $\Lambda_{m}(\Omega)$ is not closed with respect to the sum of its elements.) A~distribution $u\in\Lambda_{m}(\Omega)$ is said to be a generalized solution to the problem \eqref{ell-equ}, \eqref{bound-cond} with right-hand sides $f\in\mathcal{D}'(\Omega)$ and $g_{1},\ldots,g_{q}\in\mathcal{D}'(\Gamma)$ if $(A,B)u=(f,g_{1},\ldots,g_{q})$ for a certain operator~\eqref{(A,B)}. As usual, $\mathcal{D}'(\Omega)$ and $\mathcal{D}'(\Gamma)$ are linear topological spaces of all distributions on $\Omega$ and $\Gamma$, resp.
Note that the image $(A,B)u$ does not depend on the above parameters $p$ and $s$ for which $u\in H^{s}_{p}(\Omega)$. This follows from the fact that a sequence $(u_{k})\subset C^{\infty}(\overline{\Omega})$ approximating $u$ in $H^{s}_{p}(\Omega)$ can be build in a way that does not depend on these parameters (see, e.g., \cite[Proof of Theorem~2.3.3, Step~5]{Triebel83}).

Let us study local properties (up to the boundary $\Gamma$) of the generalized solutions to the elliptic problem. To this end, we introduce local versions of distribution spaces appearing in~\eqref{(A,B)}. Let $U$ be an open set in $\mathbb{R}^{n}$ such that $\Omega_{0}:=\Omega\cap U\neq\emptyset$, and put $\Gamma_{0}:=\Gamma\cap U$ (the $\Gamma_{0}=\emptyset$ case is admissible). Given $\sigma\in\mathbb{R}$,  $\varphi\in\Upsilon$, and $p\in(1,\infty)$, we let $H^{\sigma,\varphi}_{p,\mathrm{loc}}(\Omega_{0},\Gamma_{0})$ denote the linear space of all distributions $u\in\mathcal{D}'(\Omega)$ such that $\chi u\in H^{\sigma,\varphi}_{p}(\Omega)$ whenever $\chi\in C^{\infty}(\overline{\Omega})$ and $\mathrm{supp}\,\chi\subset\Omega_{0}\cup\Gamma_{0}$. Analogously, $B^{\sigma,\varphi}_{p,\mathrm{loc}}(\Gamma_{0})$ denotes the linear space of all distributions $v\in\mathcal{D}'(\Gamma)$ such that $\chi v\in B^{\sigma,\varphi}_{p}(\Gamma)$ whenever $\chi\in C^{\infty}(\Gamma)$ and $\mathrm{supp}\,\chi\subset\Gamma_{0}$. The operator of the multiplication by any function $\chi\in C^{\infty}(\overline{\Omega})$ (resp., $\chi\in C^{\infty}(\Gamma)$) is bounded on  $H^{\sigma,\varphi}_{p}(\Omega)$ (resp., $B^{\sigma,\varphi}_{p}(\Gamma)$), which is known in the $\varphi(\cdot)\equiv1$ case and then extends by interpolation (Theorems \ref{th-C-interp-Sobolev-got}, \ref{th-C-interp-Besov-got} and Propositions \ref{prop-quadratic-interp}, \ref{prop-quadratic-interp-Gamma}) to the $\varphi\in\Upsilon$ case. Hence, $H^{\sigma,\varphi}_{p}(\Omega)\subset
H^{\sigma,\varphi}_{p,\mathrm{loc}}(\Omega_{0},\Gamma_{0})$ and
$B^{\sigma,\varphi}_{p}(\Gamma)\subset B^{\sigma,\varphi}_{p,\mathrm{loc}}(\Gamma_{0})$
If $\Omega_{0}=\Omega$ and $\Gamma_{0}=\Gamma$, then these inclusions become equalities.

\begin{theorem}\label{th-loc-regularity}
Suppose that a distribution $u\in\Lambda_{m}(\Omega)$ is a generalized solution to the elliptic problem \eqref{ell-equ}, \eqref{bound-cond} whose right-hand sides satisfy the following conditions:
\begin{equation}\label{cond-loc-regularity}
f\in H^{s-2q,\varphi}_{p,\mathrm{loc}}(\Omega_{0},\Gamma_{0})
\;\;\mbox{and}\;\;
g_{j}\in B^{s-m_{j}-1/p,\varphi}_{p,\mathrm{loc}}(\Gamma_{0})\;\;
\mbox{whenever}\;\;j\in\{1,\ldots,q\}
\end{equation}
for certain indexes $p\in(1,\infty)$, $s>m+1/p$, and $\varphi\in\Upsilon$. Then $u\in H^{s,\varphi}_{p,\mathrm{loc}}(\Omega_{0},\Gamma_{0})$.
\end{theorem}

It follows directly from the boundedness of operator \eqref{(A,B)} that conditions \eqref{cond-loc-regularity} are also necessary for the inclusion $u\in H^{s,\varphi}_{p,\mathrm{loc}}(\Omega_{0},\Gamma_{0})$. If $\Omega_{0}=\Omega$ and $\Gamma_{0}=\Gamma$, then Theorem~\ref{th-loc-regularity} gives sufficient conditions for certain global regularity of the solution $u$, i.e. for the regularity in the whole domain $\Omega$ up to its boundary $\Gamma$. If
$\Gamma_{0}=\emptyset$, this theorem concerns the regularity inside $\Omega$ and means that
\begin{equation*}
f\in H^{s-2q,\varphi}_{p,\mathrm{loc}}(\Omega_{0},\emptyset)
\;\Longrightarrow\;
u\in H^{s,\varphi}_{p,\mathrm{loc}}(\Omega_{0},\emptyset).
\end{equation*}

\begin{proof}[Proof of Theorem $\ref{th-loc-regularity}$]  We first  consider the case where $\Omega_{0}=\Omega$ and $\Gamma_{0}=\Gamma$. In this case, the theorem is known if $\varphi(\cdot)\equiv1$ (see, e.g., \cite[Theorem~5.2 and Corollary~1.2]{Johnsen96}). Dealing with arbitrary  $\varphi\in\Upsilon$, we choose a number $\varepsilon>0$ so that $s-\varepsilon>m+1/p$. Since $(f,g)\in\Xi^{s-\varepsilon-2q}_{p}(\Omega,\Gamma)$ by \eqref{cond-loc-regularity}, we have $u\in H^{s-\varepsilon}_{p}(\Omega)$. Then
\begin{align*}
(f,g)&=(A,B)u\in\Xi^{s-2q,\varphi}_{p}(\Omega,\Gamma)\cap
(A,B)(H^{s-\varepsilon}_{p}(\Omega))\\
&=\bigl(M\dotplus(A,B)(H^{s,\varphi}_{p}(\Omega))\bigr)\cap
(A,B)(H^{s-\varepsilon}_{p}(\Omega))=(A,B)(H^{s,\varphi}_{p}(\Omega))
\end{align*}
due to \eqref{M+range}; i.e., $(f,g)=(A,B)v$ for certain $v\in H^{s,\varphi}_{p}(\Omega)$. Hence, the distribution $u-v\in H^{s-\varepsilon}_{p}(\Omega)$ satisfies $(A,B)(u-v)=0$, which yields $w:=u-v\in C^{\infty}(\overline{\Omega})$ by Theorem~\ref{th-Fredholm}. Thus, $u=v+w\in H^{s,\varphi}_{p}(\Omega)$. We have proved Theorem~\ref{th-loc-regularity} in the case where $\Omega_{0}=\Omega$ and $\Gamma_{0}=\Gamma$, which will help to treat the general case.

Let us consider the general case. Put
\begin{equation*}
\Psi:=\bigl\{\chi\in C^{\infty}(\overline{\Omega}):
\mathrm{supp}\,\chi\subset\Omega_0\cup\Gamma_0\bigr\}.
\end{equation*}
Since $u\in\Lambda_{m}(\Omega)$, we have $u\in H^{s_0}_{p_0}(\Omega)$ for some $p_0>1$ and $s_0>m+1/p_0$. Let us prove the following implication:
\begin{equation}\label{2f34}
\begin{aligned}
&\bigl(\chi u\in H^{s,\varphi}_{p}(\Omega)+
H^{s_{0}+i-1}_{p_0}(\Omega)\;\;\mbox{for every}\;\;\chi\in\Psi\bigr)\\
&\Longrightarrow\;\bigl(\chi u\in H^{s,\varphi}_{p}(\Omega)+
H^{s_{0}+i}_{p_0}(\Omega)\;\;\mbox{for every}\;\;\chi\in\Psi\bigr).
\end{aligned}
\end{equation}
Here, $i$ is an arbitrary natural number, and we use algebraic sums of spaces.

Suppose that premise of the application is true.  We consider an arbitrary function $\chi\in\Psi$ and choose a function $\eta\in\Psi$ so that $\eta=1$ in the neighbourhood of $\mathrm{supp}\,\chi$. By hypothesis \eqref{cond-loc-regularity}, we have the inclusion $\chi F\in\Xi^{s-2q,\varphi}_{p}(\Omega,\Gamma)$ where $F:=(A,B)u$. Permuting the differential operator $(A,B)$ and the operator of the multiplication by $\chi$, we obtain
\begin{equation*}
\chi F=\chi(A,B)(\eta u)=(A,B)(\chi\eta u)-(A',B')(\eta u).
\end{equation*}
Here, $(A',B')$ is a differential operator of the form $(A,B)$ with infinitely smooth coefficients such that the orders of all its components are smaller (at least by $1$) than the orders of the corresponding components of the operator $(A,B)$. (Of course, if $\mathrm{ord}\,B_{j}=0$ [i.e., $B_{j}$ is an operator of the multiplication by a function of class $C^{\infty}(\Gamma)$], then the corresponding operator $B_{j}'=0$ and we may assume that $\mathrm{ord}\,B_{j}'=-1$.)

Thus,
\begin{equation}\label{(A,B)-chi-permuted}
(A,B)(\chi u)=\chi F+(A',B')(\eta u).
\end{equation}
By the premise of implication \eqref{2f34}, we have $\eta u=u_{1}+u_{2}$ for certain distributions $u_{1}\in H^{s,\varphi}_{p}(\Omega)$ and $u_{2}\in H^{s_{0}+i-1}_{p_0}(\Omega)$. Hence,
$(A,B)(\chi u)=F_{1}+F_{2}$, where
\begin{gather}\label{2f36}
F_{1}:=\chi F+(A',B')u_{1}\in\Xi^{s-2q,\varphi}_{p}(\Omega,\Gamma),\\
F_{2}:=(A',B')u_{2}\in\Xi^{s_0+i-2q}_{p_0}(\Omega,\Gamma). \label{2f37}
\end{gather}
As to these inclusions, note that the mapping $v\mapsto(A',B')v$ sets  a bounded operator
\begin{equation}\label{(A,B)prime}
(A',B'):H^{s,\varphi}_{p}(\Omega)\to H^{s-(2q-1),\varphi}_{p}(\Omega)\times
\prod_{j=1}^{q}B^{s-(m_{j}-1)-1/p,\varphi}_{p}(\Gamma)=
\Xi^{s+1-2q,\varphi}_{p}(\Omega,\Gamma)
\end{equation}
whatever $p>1$, $s>m-1+1/p$, and $\varphi\in\Upsilon$. This is deduced from the $\varphi(\cdot)\equiv1$ case by interpolation in an analogous way to that used for \eqref{trace-operator} in the proof of Theorem~\ref{th-trace}, with the condition $s>m-1+1/p$ being stipulated by the fact that the orders of all components of $B'$ are less than or equal to $m-1$.

Since
\begin{equation*}
\Xi^{s_0+i-2q}_{p_0}(\Omega,\Gamma)= M\dotplus(A,B)(H^{s_0+i}_{p_0}(\Omega))
\end{equation*}
for a certain finite-dimensional space $M\subset C^{\infty}(\overline{\Omega})\times(C^{\infty}(\Gamma))^{l}$ (see \eqref{th-Fredholm}) and because of \eqref{2f37}, we have $F_{2}=F_{0}+(A,B)v_{2}$ for some $F_{0}\in M$ and $v_{2}\in H^{s_0+i}_{p_0}(\Omega)$. Hence,
\begin{equation*}
(A,B)(\chi u)=F_{1}+F_{0}+(A,B)v_{2},
\end{equation*}
which gives
\begin{equation}\label{proof-regularity-inclusion}
(A,B)(\chi u-v_{2})=F_{1}+F_{0}\in\Xi^{s-2q,\varphi}_{p}(\Omega,\Gamma)
\end{equation}
due to \eqref{2f36}. The right-hand side of this equality is well-defined because $\chi u-v_{2}\in H^{s_0}_{p_0}(\Omega)$. It is follows from \eqref{proof-regularity-inclusion} that
\begin{equation*}
v_{1}:=\chi u-v_{2}\in H^{s,\varphi}_{p}(\Omega),
\end{equation*}
as was shown in the first paragraph of this proof. Thus,
\begin{equation*}
\chi u=v_{1}+v_{2}\in H^{s,\varphi}_{p}(\Omega)+H^{s_0+i}_{p_0}(\Omega).
\end{equation*}
We have proved the required implication \eqref{2f34}.

Since $u\in H^{s_0}_{p_0}(\Omega)$, the premise of this implication holds true in the $i=1$ case. Hence,
$\chi u\in H^{s,\varphi}_{p}(\Omega)+
H^{s_{0}+i}_{p_0}(\Omega)$ for every $\chi\in\Psi$ and each integer $i\geq1$. As is known \cite[Theorem 4.6.2(a)]{Triebel95},
\begin{equation*}
H^{s_{0}+i}_{p_0}(\Omega)\subset H^{s-1/2}_{p}(\Omega)
\quad\mbox{whenever}\quad
s_{0}+i-\frac{n}{p_0}\geq s-\frac{1}{2}-\frac{n}{p_0}.
\end{equation*}
Hence, $H^{s_{0}+i}_{p_0}(\Omega)\subset H^{s,\varphi}_{p}(\Omega)$ if $i\gg1$. Thus, $\chi u\in H^{s,\varphi}_{p}(\Omega)$ for every $\chi\in\Psi$, q.e.d.
\end{proof}

We supplement the last theorem with the corresponding \textit{a~priori} estimate of the solution~$u$.

\begin{theorem}\label{th-loc-estimate}
Suppose that the hypotheses of Theorem~$\ref{th-loc-regularity}$ are satisfied. We arbitrarily choose functions $\chi,\eta\in C^{\infty}(\overline{\Omega})$ such that their supports lie in $\Omega_{0}\cup\Gamma_{0}$ and that $\eta(x)=1$ in a neighbourhood of $\mathrm{supp}\,\chi$. Then
\begin{equation}\label{loc-estimate}
\|\chi u,H^{s,\varphi}_{p}(\Omega)\|\leq
c\biggl(\|\chi f,H^{s-2q,\varphi}_{p}(\Omega)\|+
\sum_{j=1}^{q}\|\chi g_{j},B^{s-m_{j}-1/p,\varphi}_{p}(\Gamma)\|+
\|\eta u,H^{s-1,\varphi}_{p}(\Omega)\|\biggr).
\end{equation}
Here, $c$ is a certain positive number that does not depend on $u$, $f$, and $g_{1},\ldots,g_{q}$.
\end{theorem}

\begin{proof}
If $\Omega_{0}=\Omega$ and $\Gamma_{0}=\Gamma$, then we may take $\chi(t)\equiv\eta(t)\equiv1$ in \eqref{loc-estimate}, which means the global \textit{a~priori} estimate of~$u$. In this case, estimate \eqref{loc-estimate} follows by Peetre's lemma \cite[Lemma~3]{Peetre61} from Theorem~\ref{th-Fredholm} (from the fact that operator \eqref{(A,B)} has finite-dimensional kernel and closed range) and the compact embedding $H^{s,\varphi}_{p}(\Omega)\hookrightarrow H^{s-1,\varphi}_{p}(\Omega)$ (in fact, the continuity of this embedding is enough). In the general case, this estimate is deduced from its global version with the help of \eqref{(A,B)-chi-permuted} where $F=(A,B)u$. Namely:
\begin{align*}
\|\chi u,H^{s,\varphi}_{p}(\Omega)\|&\leq c_1
\bigl(\|(A, B)(\chi u),\Xi^{s-2q,\varphi}_{p}(\Omega,\Gamma)\|+
\|\chi u,H^{s-1,\varphi}_{p}(\Omega) \|\bigr)\\
&\leq c_1\bigl(\|\chi(A, B)u,\Xi^{s-2q,\varphi}_{p}(\Omega,\Gamma)\|+
\|(A',B')(\eta u),\Xi^{s-2q,\varphi}_{p}(\Omega,\Gamma)\|\bigr)\\
&+c_1\|\chi u,H^{s-1,\varphi}_{p}(\Omega)\|,
\end{align*}
where
\begin{equation*}
\|(A',B')(\eta u),\Xi^{s-2q,\varphi}_{p}(\Omega,\Gamma)\|
\leq c_{2}\|\eta u,H^{s-1,\varphi}_{p}(\Omega)\|
\end{equation*}
(see \eqref{(A,B)prime}) and
\begin{equation*}
\|\chi u,H^{s-1,\varphi}_{p}(\Omega)\|=
\|\chi\eta u,H^{s-1,\varphi}_{p}(\Omega)\|\leq
c_{3}\|\eta u,H^{s-1,\varphi}_{p}(\Omega)\|.
\end{equation*}
Here, the positive numbers $c_{1}$, $c_{2}$, and $c_{3}$ do not depend on $u$ (and hence on $f$ and $g_{1},\ldots,g_{q}$). This gives the required estimate \eqref{loc-estimate}.
\end{proof}

Completing this section, we apply Theorem~\ref{th-loc-regularity} to find a new condition for a solution $u$ of the elliptic problem to be $\ell$ times continuously differentiable on the set $\Omega_{0}\cup\Gamma_{0}$, i.e. for $u\in C^{\ell}(\Omega_{0}\cup\Gamma_{0})$.

\begin{theorem}\label{th-classical-smoothness}
Let $p\in(1,\infty)$, $0\leq\ell\in\mathbb{Z}$, and $\ell>m-(n-1)/p$. Suppose that a distribution $u\in\Lambda_{m}(\Omega)$ is a generalized solution to the elliptic problem \eqref{ell-equ}, \eqref{bound-cond} whose right-hand sides satisfy the following conditions:
\begin{equation}\label{cond-smoothness-data}
f\in H^{\ell-2q+n/p,\varphi}_{p,\mathrm{loc}}(\Omega_{0},\Gamma_{0})
\;\;\mbox{and}\;\;
g_{j}\in B^{\ell-m_{j}+(n-1)/p,\varphi}_{p,\mathrm{loc}}(\Gamma_{0})\;\;
\mbox{whenever}\;\;j\in\{1,\ldots,q\}
\end{equation}
for a certain function parameter $\varphi\in\Upsilon$ subject to \eqref{integral-condition}. Then $u\in C^{\ell}(\Omega_{0}\cup\Gamma_{0})$.
\end{theorem}

The proof is based on the following result:

\begin{proposition}\label{prop-embedding-in-C}
Let $1\leq n\in\mathbb{Z}$, $p\in(1,\infty)$, $0\leq\ell\in\mathbb{Z}$, and $\varphi\in\Upsilon$. Then the following assertions hold true:
\begin{itemize}
\item[(i)] Condition \eqref{integral-condition} implies the continuous embedding $H^{\ell+n/p,\varphi}_{p}(\mathbb{R}^{n})\hookrightarrow C^{\ell}_{\mathrm{b}}(\mathbb{R}^{n})$.
\item[(ii)] Suppose that $V$ is an open nonempty subset of $\mathbb{R}^{n}$. The embedding
\begin{equation}\label{H-C-embedding-local}
\bigl\{w\in H^{\ell+n/p,\varphi}_{p}(\mathbb{R}^{n}):
\mathrm{supp}\,w\subset V\bigr\}\subset C^{\ell}(\mathbb{R}^{n})
\end{equation}
implies condition \eqref{integral-condition}.
\end{itemize}
\end{proposition}

Here, $C^{\ell}_{\mathrm{b}}(\mathbb{R}^{n})$ stands for the Banach space of $\ell$ times continuously differentiable functions on $\mathbb{R}^{n}$ whose partial derivatives up to the order $\ell$ are all bounded on $\mathbb{R}^{n}$.

Assertion (i) in the $\ell=0$ case follows from \cite[Theorem~5.3]{KalyabinLizorkin87} with regard to \cite[Theorem~3.4]{CobosFernandez88} (see also \cite[Theorem~1]{Kalyabin81MathMotes}, where the proof is given). The $\ell\geq1$ case is reduced to the previous one because the differential operator $\partial/\partial x_{j}$ acts continuously from $H^{s,\varphi}_{p}(\mathbb{R}^{n})$ to $H^{s-1,\varphi}_{p}(\mathbb{R}^{n})$ for every $s\in\mathbb{R}$. Assertion (ii) is proved by a modification of the reasoning given in
\cite[Proof of Theorem~13.2]{VolevichPaneah65} in the $\ell=0$ case.
Namely, we may assume that $0\in V$; then \eqref{H-C-embedding-local} implies the continuous embedding $H^{\ell+n/p,\varphi}_{p}(V_{0})\hookrightarrow C^{\ell}(\overline{V_{0}})$ for a certain ball $V_{0}$ centered at $0$, which gives the estimate
\begin{equation*}
|\partial^{\alpha}w(0)|\leq\mathrm{const}\,
\|w,H^{\ell+n/p,\varphi}_{p}(\mathbb{R}^{n})\|
\end{equation*}
for every $w\in\mathcal{S}(\mathbb{R}^{n})$ and each partial derivative $\partial^{\alpha}$ of the order $|\alpha|\leq\ell$. Hence, the delta-function $\delta$ and all its partial derivatives $\partial^{\alpha}\delta$ with $|\alpha|\leq\ell$ belong to the dual of $H^{\ell+n/p,\varphi}_{p}(\mathbb{R}^{n})$. Since this dual equals $H^{-\ell-n/p,1/\varphi}_{p'}(\mathbb{R}^{n})$ \cite[Section~13, Subsection~3]{VolevichPaneah65}, we have the inclusion
$\partial^{\alpha}\Phi\delta=\Phi(\partial^{\alpha}\delta)\in H^{-\ell-n/p}_{p'}(\mathbb{R}^{n})$ whatever $0\leq|\alpha|\leq\ell$, with $\Phi w:=
\mathcal{F}^{-1}[\varphi^{-1}(\langle\xi\rangle)\mathcal{F}w(\xi)]$ whenever $w\in\mathcal{S}'(\mathbb{R}^{n})$. Hence, $\Phi\delta\in H^{-n/p}_{p'}(\mathbb{R}^{n})$ (see, e.g., \cite[Theorem~2.3.8(ii)]{Triebel83}), which gives $\delta\in H^{-n/p,1/\varphi}_{p'}(\mathbb{R}^{n})$. The last inclusion means that the function $\mu(\xi):=\langle\xi\rangle^{n/p}\varphi(\langle\xi\rangle)$ of $\xi\in\mathbb{R}^{n}$ satisfies $\mathcal{F}^{-1}\mu^{-1}(\xi)\in L_{p'}(\mathbb{R}^{n})$. This inclusion is equivalent to \eqref{integral-condition} due to \cite[Theorem~13.2]{VolevichPaneah65} and \cite[Theorem~5.3]{KalyabinLizorkin87}. Thus, \eqref{H-C-embedding-local} implies \eqref{integral-condition}.

\begin{proof}[Proof of Theorem~$\ref{th-classical-smoothness}$.]
Choosing a point $x\in\Omega_{0}\cup\Gamma_{0}$ arbitrarily, we take a function $\chi\in C^{\infty}(\overline{\Omega})$ such that $\mathrm{supp}\,\chi\subset\Omega_0\cup\Gamma_0$ and that $\chi=1$
in a neighbourhood $V(x)$ of $x$ (in the topology on $\overline{\Omega}$). Note that $s:=\ell+n/p>m+1/p$ by the hypothesis. Hence, owing to Theorem~\ref{th-loc-regularity}, the inclusion $\chi u\in H^{\ell+n/p,\varphi}_{p}(\Omega)$ holds true. Therefore, $\chi u$ has an extension
$w\in H^{\ell+n/p,\varphi}_{p}(\mathbb{R}^{n})\subset C^{\ell}(\mathbb{R}^{n})$ by hypothesis \eqref{integral-condition} and Proposition~\ref{prop-embedding-in-C}. Thus, $u\in C^{\ell}(V(x))$, which implies that $u\in C^{\ell}(\Omega_{0}\cup\Gamma_{0})$ due to the arbitrariness of $x\in\Omega_{0}\cup\Gamma_{0}$.
\end{proof}

\begin{remark}
Condition \eqref{integral-condition} is exact in Theorem~\ref{th-classical-smoothness}. Namely, it follows from the implication
\begin{equation}\label{implication-u}
\bigl(u\in\Lambda_{m}(\Omega)\;\,\mbox{and}\;\,
(f,g):=(A,B)u\;\,\mbox{satisfies}\;\,\eqref{cond-smoothness-data}\bigr)
\;\Longrightarrow\;u\in C^{\ell}(\Omega_{0}\cup\Gamma_{0})
\end{equation}
that $\varphi$ is subject to \eqref{integral-condition}. Here, $p$ and $\ell$ satisfy the hypotheses of the theorem, whereas $\varphi\in\Upsilon$ is chosen arbitrarily. Indeed, assuming \eqref{implication-u} to be valid, we consider an open ball $V\subset\Omega_{0}$ in $\mathbb{R}^{n}$
and arbitrarily choose a distribution $w\in H^{\ell+n/p,\varphi}(\mathbb{R}^{n})$ such that $\mathrm{supp}\,w\subset V$. Since the distribution $u:=w\!\upharpoonright\Omega\!\in H^{\ell+n/p,\varphi}(\Omega)$ satisfies the premise of implication \eqref{implication-u}, we have the inclusion $u\in C^{\ell}(\Omega_{0}\cup\Gamma_{0})$, which means that $w\in C^{l}(\mathbb{R}^{n})$. Thus, \eqref{H-C-embedding-local} holds true, which entails \eqref{integral-condition} by Proposition~\ref{prop-embedding-in-C}.
\end{remark}

If we remained in the framework of Sobolev spaces, we would have to change \eqref{cond-smoothness-data} for the rougher condition
\begin{equation*}
f\in H^{\ell-2q+n/p+\varepsilon}_{p,\mathrm{loc}}(\Omega_{0},\Gamma_{0})
\;\;\mbox{and}\;\;
g_{j}\in B^{\ell-m_{j}+(n-1)/p+\varepsilon}_{p,\mathrm{loc}}(\Gamma_{0})\;\;
\mbox{whenever}\;\;j\in\{1,\ldots,q\}
\end{equation*}
for some $\varepsilon>0$ to ensure the inclusion $u\in C^{\ell}(\Omega_{0}\cup\Gamma_{0})$. The use of the supplementary function parameter $\varphi\in\Upsilon$ allows attaining the limiting values of the number parameters in \eqref{cond-smoothness-data}.

\section{Applications to parameter-elliptic boundary-value problems}\label{sec4}

As in the previous section, integers $q\geq1$ and $m_{1},\ldots,m_{q}\geq0$ are arbitrarily chosen, and $m:=\max\{m_{1},\ldots,m_{q}\}$. We consider a boundary-value problem
\begin{gather} \label{4.24}
A(\lambda)\,u=f\quad\mbox{in}\quad\Omega,\\
B_{j}(\lambda)\,u=g_{j}\quad\mbox{on}\quad\Gamma,\quad j=1,\ldots,q,
\label{4.25}
\end{gather}
that depends on the complex-valued parameter $\lambda$ in the following way:
\begin{equation*}
A(\lambda):=\sum_{r=0}^{2q}\,\lambda^{2q-r}A_{r}\quad\;\mbox{and}\quad\;
B_{j}(\lambda):=\sum_{r=0}^{m_{j}}\,\lambda^{m_{j}-r}B_{j,r}.
\end{equation*}
Here, each $A_{r}$ is a linear PDO on $\overline{\Omega}$ with $\mathrm{ord}\,A_{r}\leq r$, and every
$B_{j,r}$ is a linear boundary PDO on $\Gamma$ with $\mathrm{ord}\,B_{j,r}\leq r$; all coefficients of these PDOs belong to $C^{\infty}(\overline{\Omega})$ and $C^{\infty}(\Gamma)$, resp. Put $B(\lambda):=(B_{1}(\lambda),\ldots,B_{q}(\lambda))$.

Let $K$ be a fixed closed angle on the complex plane with vertex at the origin (the case where $K$ degenerates into a ray is admissible). We suppose that the boundary-value problem \eqref{4.24} is parameter-elliptic in $K$, i.e. it satisfies the following two conditions (see, e.g. \cite[Subsection~9.1.1]{Roitberg96}):
\begin{itemize}
\item [(a)] For all $x\in\overline{\Omega}$,
$\xi\in\mathbb{R}^{n}$ and $\lambda\in K$ such that $|\xi|+|\lambda|\neq0$,
the inequality $A^{\circ}(x;\xi,\lambda)\neq0$ holds true.
\item [(b)] For arbitrarily fixed point $x\in\Gamma$, vector
$\xi\in\nobreak\mathbb{R}^{n}$, tangent to the boundary $\Gamma$ at the point $x$, and the parameter $\lambda\in K$ such that $|\xi|+|\lambda|\neq0$, the polynomials $B^{\circ}_{j}(x;\xi+\tau\nu(x),\lambda)$, $j=1,\ldots,q$, in $\tau\in\mathbb{C}$ are linearly independent modulo the polynomial $\prod_{j=1}^{q}(\tau-\tau^{+}_{j}(x;\xi,\lambda))$. Here,
$\tau^{+}_{1}(x;\xi,\lambda),\ldots,\tau^{+}_{q}(x;\xi,\lambda)$ are all $\tau$-roots of the polynomial $A^{\circ}(x;\xi+\tau\nu(x),\lambda)$ that have the positive imaginary part and are written with regard for their multiplicity, whereas $\nu(x)$ stands for the unit vector of the inner normal to $\Gamma$ at~$x$.
\end{itemize}
We have used the following notations:
\begin{equation*}
A^{\circ}(x;\xi,\lambda):=
\sum_{r=0}^{2q}\,\lambda^{2q-r}A^{\circ}_{r}(x,\xi)
\quad\mbox{for all}\quad
x\in\overline{\Omega},\;\xi\in\mathbb{C}^{n},\;\lambda\in\mathbb{C},
\end{equation*}
where $A^{\circ}_{r}(x,\xi)$ is the principal symbol of the PDO $A_{r}$ if $\mathrm{ord}\,A_{r}=r$, otherwise  $A^{\circ}_{r}(x,\xi)\equiv0$, and
\begin{equation*}
B^{\circ}_{j}(x;\xi,\lambda):=
\sum_{r=0}^{m_{j}}\,\lambda^{m_{j}-r}B^{\circ}_{j,r}(x,\xi)
\quad\mbox{for all}\quad x\in\Gamma,\;\xi\in \mathbb{C}^{n},\;\lambda\in\mathbb{C},
\end{equation*}
where $B^{\circ}_{j,r}(x,\xi)$ is the principal symbol of the boundary PDO $B_{j,r}$ if $\mathrm{ord}\,B_{j,r}=r$, otherwise  $B^{\circ}_{j,r}(x,\xi)\equiv0$. Thus, $A^{\circ}(x;\xi,\lambda)$ and $B^{\circ}_{j}(x;\xi,\lambda)$ are homogeneous polynomials in $(\xi,\lambda)$ of the orders $2q$ and $m_{j}$, resp.

Conditions (a) and (b) in the $\lambda=0$ case mean that the boundary-value problem \eqref{4.24}, \eqref{4.25} is elliptic for fixed $\lambda=0$. Since $\lambda$ affects only the lower order terms of the PDOs $A(\lambda)$ and $B_{j}(\lambda)$, this problem is elliptic for all $\lambda\in\mathbb{C}$. According to Theorem~\ref{th-Fredholm}, the mapping
\begin{equation*}
u\mapsto(A(\lambda)u,B_{1}(\lambda)u,
\ldots,B_{q}(\lambda)u)=(A(\lambda)u,B(\lambda)u),\quad\mbox{where}\quad
u\in C^{\infty}(\overline{\Omega}),
\end{equation*}
extends uniquely (by continuity) to a Fredholm bounded operator.
\begin{equation}\label{4.27}
(A(\lambda),B(\lambda)):H^{s,\varphi}_{p}(\Omega)\rightarrow
\Xi^{s-2q,\varphi}_{p}(\Omega,\Gamma)
\end{equation}
for arbitrary $p\in(1,\infty)$, $s>m+1/p$, $\varphi\in\Upsilon$, and
$\lambda\in\mathbb{C}$. The kernel and index of \eqref{4.27} are independent of $p$, $s$, and $\varphi$. Owing to \cite[Theorem 20.1.8]{Hermander07v3}, the index does not depend on $\lambda$ in the case where $p=2$ and $\varphi(\cdot)\equiv1$. Hence, the index is independent of $\lambda$ in the general case.

Since the boundary-value problem \eqref{4.24}, \eqref{4.25} is parameter-elliptic in the angle $K$, operator \eqref{4.27} has additional important properties.

\begin{theorem}\label{th4.9}
The following assertions are true:
\begin{itemize}
\item[$\mathrm{(i)}$] There exists a number $\lambda_{0}>0$ such that, for each  $\lambda\in K$ with $|\lambda|\geq\nobreak\lambda_{0}$, operator \eqref{4.27} is an isomorphism of $H^{s,\varphi}_{p}(\Omega)$ onto $\Xi^{s,\varphi}_{p}(\Omega,\Gamma)$ whatever $p\in(1,\infty)$, $s>m+1/p$, and $\varphi\in\Upsilon$.
\item[$\mathrm{(ii)}$] For arbitrarily fixed parameters $p\in(1,\infty)$, $s>\max\{2q,m+1/p\}$, and $\varphi\in\Upsilon$, there exists a number $c=c(p,s,\varphi)\geq1$ such that, for each $\lambda\in K$ with $|\lambda|\geq\max\{\lambda_{0},1\}$ and for every distribution $u\in H^{s,\varphi}_{p}(\Omega)$, we have the two-sided estimate
\begin{align}
&c^{-1}\bigl(\|u,H^{s,\varphi}_{p}(\Omega)\|+
|\lambda|^{s}\varphi(|\lambda|)\,\|u,L_{p}(\Omega)\|\bigr) \notag \\
&\leq \|A(\lambda)u,H^{s-2q,\varphi}_{p}(\Omega)\|+
|\lambda|^{s-2q}\varphi(|\lambda|)\,\|A(\lambda)u,L_{p}(\Omega)\| \notag
\\ \notag
&\quad+\sum_{j=1}^{q}\,\bigl(\|B_{j}(\lambda)u, B^{s-m_{j}-1/p,\varphi}_{p}(\Gamma)\| +|\lambda|^{s-m_{j}-1/p}\varphi(|\lambda|)\,
\|B_{j}(\lambda)u,B^{0}_{p}(\Gamma)\|\bigr)\\
&\leq c\,\bigl(\|u,H^{s,\varphi}_{p}(\Omega)\|+
|\lambda|^{s}\varphi(|\lambda|)\|u,L_{p}(\Omega)\|\bigr). \label{4.28}
\end{align}
Here, the number $c$ does not depend on $u$ and $\lambda$.
\end{itemize}
\end{theorem}

Note that $L_{p}(\Omega)=H^{0}_{p}(\Omega)$  with equality of norms. Since the ellipticity of the parameter-dependent problem \eqref{4.24}, \eqref{4.25} in $K$ implies its ellipticity in the angle $-K:=\{-\lambda:\lambda\in K\}$, the conclusions of this theorem is also valid for $\lambda\in-K$ with $|\lambda|\gg1$.

We separately prove assertions (i) and (ii) of Theorem \ref{th4.9}.

\begin{proof}[Proof of assertion \rm(i)\it]
This assertion is known in the case where $p=2$, $s\geq2q$, $s>m+1/2$, and $\varphi(\cdot)\equiv1$ (see, e.g., \cite[Theorem~3.2.1]{Agranovich97}). Thus, the assertion holds true in the general case because the kernel and index of operator \eqref{4.27} are independent of $p\in(1,\infty)$, $s>m+1/p$, and $\varphi\in\Upsilon$.
\end{proof}

To prove assertion (ii), it is useful to introduce some distribution spaces whose norms depend on an additional parameter $\varrho\geq1$. Such norms are present in the two-sided estimate \eqref{4.28} with $\varrho=|\lambda|$.

Given $s>0$, $\varphi\in\Upsilon$, $p\in(1,\infty)$, and $\varrho\geq1$, we let $H^{s,\varphi}_{p}(\Omega,\varrho)$ denote the space $H^{s,\varphi}_{p}(\Omega)$ endowed with the equivalent norm
\begin{equation*}
\|u,H^{s,\varphi}_{p}(\Omega,\varrho)\|:= \|u,H^{s,\varphi}_{p}(\Omega)\|+
\varrho^{s}\varphi(\varrho)\|u,L_{p}(G)\|.
\end{equation*}
Similarly, we let $B^{s,\varphi}_{p}(G,\varrho)$ denote the space $B^{s,\varphi}_{p}(G)$ endowed with the equivalent norm
\begin{equation*}
\|v,B^{s,\varphi}_{p}(G,\varrho)\|:= \|v,B^{s,\varphi}_{p}(G)\|+
\varrho^{s}\varphi(\varrho)\|v,B^{0}_{p}(G)\|.
\end{equation*}
Here, $\Omega$ satisfies \eqref{Omega-assumption}, and $G=\mathbb{R}^{n}$, with $1\leq n\in\mathbb{Z}$, or $G=\Gamma$.

We deduce assertion~(ii) of Theorem~\ref{th4.9} by interpolation and rely on the following result.

\begin{lemma}\label{lemma4.2}
Suppose that the hypotheses of Theorem~$\ref{th-C-interp-Sobolev}$ are satisfied and that $s_{0}>0$ and $s_1>0$. Then there exists a number $c\geq1$ such that
\begin{equation}\label{I-3}
c^{-1}\|u,H^{s,\varphi}_{p}(\Omega,\varrho)\|\leq
\|u,[H^{s_{0},\varphi_{0}}_{p_0}(\Omega,\varrho), H^{s_{1},\varphi_{1}}_{p_1}(\Omega,\varrho)]_{\theta}\|\leq c\,\|u,H^{s,\varphi}_{p}(\Omega,\varrho)\|
\end{equation}
and
\begin{equation}\label{I-4}
c^{-1}\|v,B^{s,\varphi}_{p}(G,\varrho)\|\leq
\|v,[B^{s_{0},\varphi_{0}}_{p_0}(G,\varrho), B^{s_{1},\varphi_{1}}_{p_{1}}(G,\varrho)]_{\theta}\|\leq c\,\|v,B^{s,\varphi}_{p}(G,\varrho)\|
\end{equation}
for all $u\in H^{s,\varphi}_{p}(\Omega,\varrho)$, $v\in B^{s,\varphi}_{p}(G,\varrho)$, and $\varrho\geq1$. The number $c$ does not depend on $u$, $v$, and $\varrho$.
\end{lemma}


\begin{proof}
We will first prove \eqref{I-4} in the case  where $G=\mathbb{R}^n$. As was mentioned in the proof of Theorem~\ref{th-C-interp-Besov}, the space $B^{s,\varphi}_{p}(\mathbb{R}^n)$ is a retract of the Banach space $l^{s,\varphi}_{p}(L_p(\mathbb{R}^n))=L_p(\mathbb{R}^n,l^{s,\varphi}_{p})$
whatever $s\in\mathbb{R}$, $\varphi\in\Upsilon$, and $1<p<\infty$. Here,  $l^{s,\varphi}_{p}$ is defined by \eqref{space-l(E)} where $E:=\mathbb{C}$, with $l_{p}=l^{s,\varphi}_{p}$ in the case when $s=0$ and $\varphi(\cdot)\equiv1$. As usual, $L_p(\mathbb{R}^n, E)$, where $E$ is a Banach space, denotes the Banach space of all Bochner measurable functions $f:\mathbb{R}^n\rightarrow E$ (with respect to the Lebesgue measure on $\mathbb{R}^n$) such that
\begin{equation*}
\|f,L_p(\mathbb{R}^n,E)\|^{p}:=\int\limits_{\mathbb{R}^n}
\|f(x),E\|^{p}dx<\infty.
\end{equation*}
The corresponding retraction
\begin{equation}\label{II-3}
\mathcal{R}:L_p(\mathbb{R}^n,l^{s,\varphi}_{p})\rightarrow B^{s,\varphi}_{p}(\mathbb{R}^n)
\end{equation}
and coretraction
\begin{equation}\label{II-4}
\mathcal{T}:B^{s,\varphi}_{p}(\mathbb{R}^n)\rightarrow L_p(\mathbb{R}^n,l^{s,\varphi}_{p})
\end{equation}
are independent of $s,\varphi$, and $p$.

Let us evaluate the norms of $\mathcal{R}$ and $\mathcal{T}$ via the norm in  $B^{s,\varphi}_{p}(\mathbb{R}^n,\varrho)$ with $\varrho\geq1$. Given $f=(f_{k})^{\infty}_{k=0}\in L_p(\mathbb{R}^n,l^{s,\varphi}_{p})$, we have
\begin{align}\label{III-1}
\begin{split}
\|\mathcal{R}f, B^{s,\varphi}_{p}(\mathbb{R}^n,\varrho)\| & = \|\mathcal{R}f, B^{s,\varphi}_{p}(\mathbb{R}^n)\|+ \varrho^{s}\varphi(\varrho)\|\mathcal{R}f, B^{0}_{p}(\mathbb{R}^n)\| \\
& \leq
c_1(\|f,L_p(\mathbb{R}^n,l^{s,\varphi}_{p})\|+ \varrho^{s}\varphi(\varrho)\|f,L_p(\mathbb{R}^n,l_{p})\|), \\
\end{split}
\end{align}
where $c_1$ is the maximum of the norms of operator \eqref{II-3} and this operator for $s=0$ and $\varphi(\cdot)=1$. Here,
\begin{align}\label{III-2}
\begin{split}
&\|f,L_p(\mathbb{R}^n,l^{s,\varphi}_{p})\|^{p}+
(\varrho^{s}\varphi(\varrho)\|f,L_p(\mathbb{R}^n,l_{p}\|)^{p}\\
= & \int\limits_{\mathbb{R}^n}\|f(x),l^{s,\varphi}_{p}\|^{p}dx + (\varrho^{s}\varphi(\varrho))^{p}
\int\limits_{\mathbb{R}^n}\|f(x),l_{p}\|^{p}dx \\
= &
\int\limits_{\mathbb{R}^n}\sum^{\infty}_{k=0}
(2^{sk}\varphi(2^{k})|f_{k}(x)|)^{p}dx + (\rho^{s}\varphi(\varrho))^{p}\int\limits_{\mathbb{R}^n}\sum^{\infty}_{k=0}
|f_{k}(x)|^{p}dx \\
= & \int\limits_{\mathbb{R}^n}
\sum^{\infty}_{k=0}\varkappa^{p}_{k}(s,\varphi,p,\varrho)|f_{k}(x)|^{p}dx =  \|f,L_p(\mathbb{R}^n,l^{s,\varphi}_{p}(\varrho))\|^{p}, \\
\end{split}
\end{align}
where
\begin{equation*}
\varkappa_{k}(s,\varphi,p,\varrho):=
((2^{sk}\varphi(2^{k}))^{p}+(\varrho^{s}\varphi(\varrho))^{p})^{1/p}
\end{equation*}
and $l^{s,\varphi}_{p}(\varrho)$ denotes the space $l^{s,\varphi}_{p}$ endowed with the equivalent norm
\begin{equation*}
\|(\alpha_{k})^{\infty}_{k=0},l^{s,\varphi}_{p}(\varrho)\|:=
\left(\sum^{\infty}_{k=0}
\varkappa^{p}_{k}(s,\varphi,p,\varrho)|\alpha_{k}|^{p}\right)^{1/p}.
\end{equation*}
According to \eqref{III-1} and \eqref{III-2}, we get
\begin{equation}\label{IV-1}
\|\mathcal{R}f,B^{s,\varphi}_{p}(\mathbb{R}^n,\varrho)\|\leq c_{1}2^{1-1/p}\,\|f,L_{p}(\mathbb{R}^n,l^{s,\varphi}_{p}(\varrho))\|.
\end{equation}

Furthermore, given $w\in B^{s,\varphi}_{p}(\mathbb{R}^n)$, we infer by \eqref{III-2} that
\begin{align*}
\|\mathcal{T}w,L_{p}(\mathbb{R}^n,l^{s,\varphi}_{p}(\varrho))\|^{p}= & \|\mathcal{T}w,L_{p}(\mathbb{R}^n,l^{s,\varphi}_{p})\|^{p}+
(\varrho^{s}\varphi(\rho)\|\mathcal{T}w,L_{p}(\mathbb{R}^n,l_{p})\|)^{p}\\
& \leq c^{p}_{2}\bigl(\|w,B^{s,\varphi}_{p}(\mathbb{R}^n)\|^{p}+
(\varrho^{s}\varphi(\rho)\|w,B^{0}_{p}(\mathbb{R}^n)\|)^{p}\bigr)\\
&=c^{p}_{2}\,\|w,B^{s,\varphi}_{p}(\mathbb{R}^n,\varrho)\|^{p},
\end{align*}
where $c_{2}$ is the maximum of the norms of operator \eqref{II-4} and this operator for $s=0$ and $\varphi(\cdot)=1$. Hence,
\begin{equation}\label{IV-2}
\|\mathcal{T}w,L_{p}(\mathbb{R}^n,l^{s,\varphi}_{p}(\varrho))\|\leq c_{2}\|w,B^{s,\varphi}_{p}(\mathbb{R}^n,\varrho)\|.
\end{equation}

We conclude by \eqref{IV-1} and \eqref{IV-2} that that the norms of the operators
\begin{equation}\label{IV-3}
\mathcal{R}:L_p(\mathbb{R}^n,l^{s,\varphi}_{p}(\varrho))\rightarrow B^{s,\varphi}_{p}(\mathbb{R}^n,\varrho),
\end{equation}
\begin{equation}\label{IV-4}
\mathcal{R}:L_{p_{j}}(\mathbb{R}^n,l^{s_{j},\varphi_{j}}_{p_{j}}(\varrho))\rightarrow B^{s_{j},\varphi_{j}}_{p_{j}}(\mathbb{R}^n,\varrho),\quad j\in\{0,1\},
\end{equation}
and
\begin{equation}\label{IV-5}
\mathcal{T}:B^{s,\varphi}_{p}(\mathbb{R}^n,\varrho)\rightarrow L_p(\mathbb{R}^n,l^{s,\varphi}_{p}(\varrho))
\end{equation}
\begin{equation}\label{IV-6}
\mathcal{T}:B^{s_{j},\varphi_{j}}_{p_{j}}(\mathbb{R}^n,\varrho)\rightarrow L_{p_{j}}(\mathbb{R}^n,l^{s_{j},\varphi_{j}}_{p_{j}}(\varrho)),
\quad j\in\{0,1\},
\end{equation}
are bounded from above by  a certain number $c$ that does not depend on $\varrho\geq1$.

Interpolating operators \eqref{IV-4} and \eqref{IV-6} and putting
\begin{equation*}
\mathcal{L}_{\theta}(\varrho):=
[L_{p_0}(\mathbb{R}^n,l^{s_{0},\varphi_{0}}_{p_{0}}(\varrho)),
L_{p_1}(\mathbb{R}^n,l^{s_{1},\varphi_{1}}_{p_{1}}(\varrho))]_{\theta}
\end{equation*}
and
\begin{equation*}
\mathcal{B}_{\theta}(\varrho):=
[B^{s_{0},\varphi_{0}}_{p_{0}}(\mathbb{R}^n,\varrho),
B^{s_{1},\varphi_{1}}_{p_{1}}(\mathbb{R}^n,\varrho)]_{\theta},
\end{equation*}
we get two bounded operators
\begin{equation}\label{V-1}
\mathcal{R}:\mathcal{L}_{\theta}(\varrho)\to\mathcal{B}_{\theta}(\varrho) \quad\mbox{and}\quad
\mathcal{T}:\mathcal{B}_{\theta}(\varrho)\to\mathcal{L}_{\theta}(\varrho),
\end{equation}
whose norms does not exceed $c$. According to \cite[Theorem 1.18.4 and Remark 1.18.4/2]{Triebel95},
\begin{equation}\label{V-2}
\mathcal{L}_{\theta}(\varrho)=
L_{p}(\mathbb{R}^n,[l^{s_{0},\varphi_{0}}_{p_{0}}(\varrho),
l^{s_{1},\varphi_{1}}_{p_{1}}(\varrho)]_{\theta})
\end{equation}
with equality of norms. Let us describe the last interpolation space.

Observe that
\begin{equation}\label{V-3}
l^{s,\varphi}_{p}(\varrho)= l_{p}((\varkappa_{k}(s,\varphi,p,\varrho)\mathbb{C})^{\infty}_{k=0}).
\end{equation}
Here, given a sequence of Banach spaces $(E_{k})^{\infty}_{k=0}$, we let $ l_{p}((E_{k})^{\infty}_{k=0})$ denote the linear space of all sequences $\alpha=(\alpha_{k})^{\infty}_{k=0}$ such that each $\alpha_{k}\in E_{k}$ and
\begin{equation*}
\|\alpha,l_{p}((E_{k})^{\infty}_{k=0})\|^{p}:=
\sum^{\infty}_{k=0}\|\alpha_{k},E_{k}\|^{p}<\infty.
\end{equation*}
This space is Banach with respect to the norm $\|\cdot,l_{p}((E_{k})^{\infty}_{k=0})\|$. We consider the case where each $E_{k}=\varkappa_{k}(s,\varphi,p,\varrho)\mathbb{C}$, i.e. $E_{k}$ is the complex plain $\mathbb{C}$ endowed with the proportional norm
\begin{equation*}
\|\beta,\varkappa_{k}(s,\varphi,p,\varrho)\mathbb{C}\|:=
\varkappa_{k}(s,\varphi,p,\varrho)|\beta|,
\end{equation*}
with $\beta\in \mathbb{C}$. Owing to \cite[Theorem 1.18.1 and Remark 1.18.1/1]{Triebel95},
\begin{align}\label{VI-1}
\begin{split}
[l^{s_0,\varphi_0}_{p_0}(\varrho),
l^{s_1,\varphi_1}_{p_1}(\varrho)]_{\theta}
=& [l_{p_0}((\varkappa_{k}(s_{0},\varphi_{0},p_{0},\varrho)
\mathbb{C})^{\infty}_{k=0}),
l_{p_1}((\varkappa_{k}(s_{1},\varphi_{1},p_{1},\varrho)
\mathbb{C})^{\infty}_{k=0})]_{\theta}\\
= & l_{p}(([\varkappa_{k}(s_{0},\varphi_{0},p_{0},\varrho)\mathbb{C},
\varkappa_{k}(s_{1},\varphi_{1},p_{1},\varrho)
\mathbb{C}]_{\theta})_{k=0}^{\infty})\\
= & l_{p}((\varkappa^{1-\theta}_{k}(s_{0},\varphi_{0},p_{0},\varrho)
\varkappa^{\theta}_{k}(s_{1},\varphi_{1},p_{1},\varrho)
\mathbb{C})_{k=0}^{\infty})
\end{split}
\end{align}
because the functor $[\cdot,\cdot]_{\theta}$ is exact of type $\theta$ (i.e. it satisfies \eqref{estimate-norms-C-interp}; see Appendix~\ref{appendix}).

According to \cite[Lemma~4.2]{MikhailetsMurach14},
\begin{align}\label{VII-1}
\begin{split}
& \varkappa^{1-\theta}_{k}(s_{0},\varphi_{0},p_{0},\varrho)
\varkappa^{\theta}_{k}(s_{1},\varphi_{1},p_{1},\varrho)\\
&=  ((2^{k})^{s_{0}p_{0}}\varphi^{p_{0}}_{0}(2^{k})+ \varrho^{s_{0}p_{0}}\varphi^{p_{0}}_{0}(\varrho))^{(1-\theta)/p_{0}}\,
((2^{k})^{s_{1}p_{1}}\varphi^{p_{1}}_{1}(2^{k})+ \varrho^{s_{1}p_{1}}\varphi^{p_{1}}_{1}(\varrho))^{\theta/p_{1}}\\
&\asymp ((2^{k}+\varrho)^{s_{0}p_{0}}
\varphi^{p_{0}}_{0}(2^{k}+\varrho))^{(1-\theta)/p_{0}}\,
((2^{k}+\varrho)^{s_{1}p_{1}}
\varphi^{p_{1}}_{1}(2^{k}+\varrho))^{\theta/p_{1}}\\
&=  (2^{k}+\varrho)^{s}\varphi(2^{k}+\varrho)= ((2^{k}+\varrho)^{sp}\varphi^{p}(2^{k}+\varrho))^{1/p}\\
& \asymp ((2^{k})^{sp}\varphi^{p}(2^{k})+\varrho^{sp}\varphi^{p}(\varrho))^{1/p}
=\varkappa_{k}(s,\varphi,p,\varrho),
\end{split}
\end{align}
the weak equivalence $\asymp$ of positive values being with respect to $k$ and $\varrho$. Namely, we write $h_1(k,\varrho)\asymp h_2(k,\varrho)$ for positive values $h_1$ and $h_2$ depending on $k\in\{0,1,2,\ldots\}$ and $\varrho\geq1$ iff there exist positive numbers $c'$ and $c''$ such that $c'\leq h_1(k,\varrho)/h_2(k,\varrho)\leq c''$ for arbitrary $k$ and $\varrho$ indicated. Explaining the last weak equivalence in \eqref{VII-1}, we note that the function $\eta(t):=t^{sp}\varphi^{p}(t)$ of $t\geq1$ satisfies $\eta(2^{k}+\varrho)\asymp\eta(2^{k})+\eta(\varrho)$ by \cite[Lemma~4.2]{MikhailetsMurach14} (because $sp>0$ and $\varphi^{p}$ varies slowly at infinity), which yields the required equivalence. The first weak equivalence is justified in the same way.

Formulas \eqref{VI-1} and \eqref{VII-1} imply the following:
\begin{equation}\label{VII-2}
[l^{s_0,\varphi_0}_{p_0}(\varrho),
l^{s_1,\varphi_1}_{p_1}(\varrho)]_{\theta}=l^{s,\varphi}_{p}(\varrho),
\end{equation}
and there exists a number $c_{0}\geq1$ such that
\begin{equation*}\label{VII-3}
c^{-1}_{0}\|\alpha,l^{s,\varphi}_{p}(\varrho)\|\leq
\|\alpha,[l^{s_0,\varphi_0}_{p_0}(\varrho),
l^{s_1,\varphi_1}_{p_1}(\varrho)]_{\theta}\|\leq c_{0}\|\alpha,l^{s,\varphi}_{p}(\varrho)\|
\end{equation*}
for all $\alpha\in l^{s,\varphi}_{p}$ and $\varrho\geq1$. We hence infer by \eqref{V-2} that
\begin{equation}\label{VII-4}
\mathcal{L}_{\theta}(\varrho)=
[L_{p_0}(\mathbb{R}^{n},l^{s_0,\varphi_0}_{p_0}(\varrho)), L_{p_1}(\mathbb{R}^{n},l^{s_1,\varphi_1}_{p_1}(\varrho))]_{\theta}
\end{equation}
and
\begin{equation}\label{VIII-1}
c^{-1}_{0}\|f,\mathcal{L}_{\theta}(\varrho)\|\leq \|f,L_{p}(\mathbb{R}^{n},l^{s,\varphi}_{p}(\varrho))\|\leq
c\|f,\mathcal{L}_{\theta}(\varrho)\|,
\end{equation}
for every  $f\in\mathcal{L}_{\theta}(\varrho)$ and $\varrho\geq1$, where $c_0$ is independent of $f$ and $\varrho$.

Now for every $w\in B^{s,\varphi}_{p}(\mathbb{R}^{n})$ we get
\begin{align*}
\|w, B^{s,\varphi}_{p}(\mathbb{R}^{n},\varrho)\|&= \|\mathcal{R}\mathcal{T}w, B^{s,\varphi}_{p}(\mathbb{R}^{n},\varrho)\|
\leq  c\|\mathcal{T}w,L_{p}(\mathbb{R}^{n},l^{s,\varphi}_{p}(\varrho))\|\\
&\leq cc_{0}\|\mathcal{T}w,\mathcal{L}_{\theta}(\varrho)\|
\leq c^{2}c_{0}\|w,B_{\theta}(\varrho)\|
\end{align*}
due to \eqref{IV-3}, \eqref{VIII-1}, and \eqref{V-1}. Moreover,
\begin{align*}
\|w, B_{\theta}(\varrho)\|&=
\|\mathcal{R}\mathcal{T}w, B_{\theta}(\varrho)\|
\leq c\|\mathcal{T}w,L_{\theta}(\varrho)\|\\
&\leq cc_{0}\|\mathcal{T}w,L_{p}(\mathbb{R}^{n},l^{s,\varphi}_{p}\varrho)\|
\leq c^{2}c_{0}\|w,B^{s,\varphi}_{p}(\mathbb{R}^{n},\varrho)\|
\end{align*}
due to \eqref{V-1}, \eqref{VIII-1}, and \eqref{IV-5}. This means the required property \eqref{I-4} in the $G=\mathbb{R}^{n}$ case.

Property \eqref{I-3} in this case is proved by the same reasoning provided we rely on the fact that $H^{s,\varphi}_{p}(\mathbb{R}^n)$ is a retract of $L_p(\mathbb{R}^n,l^{s,\varphi}_{2})$ \cite[Theorem~2.5]{CobosFernandez88} and hence use only $l^{s,\varphi}_{2}$ spaces instead of $l^{s,\varphi}_{p}$ and  $l^{s,\varphi}_{p_j}$. (Other way to prove \eqref{I-3} is to adapt the consideration given in  \cite[Section~1.1]{GrubbKokholm93}.)

The other cases for $\Omega$ and $G$ are deduced from the case just considered by the same reasoning as that used in \cite[Proof of Lemma~4.1]{MikhailetsMurach14} for quadratic interpolation between
parameter-dependent Sobolev spaces with the integrability exponent $p=2$.
\end{proof}

Now we may complete the proof of Theorem~\ref{th4.9}.

\begin{proof}[Proof of assertion \rm(ii)\it]
If $\varphi(\cdot)\equiv1$, then this assertion is contained in Roitberg's result \cite[Theorem 9.1.1]{Roitberg96} (the $p=2$ case is studied by Agranovich and Vishik \cite[Theorem~6.1]{AgranovichVishik64} and
\cite[Theorem~3.2.1]{Agranovich97}; as to pseudodifferential parameter-elliptic problems, see Grubb and Kokholm's article \cite[Corollary~4.7]{GrubbKokholm93} and also \cite[Theorem~1.9]{Grubb95}). This result concerns the operator
\begin{equation*}
(A(\lambda),B(\lambda)):\widetilde{H}^{s,p,(r)}(\Omega,|\lambda|)\to
\widetilde{H}^{s-2q,p,(r-2q)}(\Omega,|\lambda|)\times
\prod_{j=1}^{q}B^{s-m_{j}-1/p}_{p}(\Gamma,|\lambda|)
\end{equation*}
where $s\in\mathbb{R}$, $p\in(1,\infty)$, and $r:=\max\{2q,m+1\}$. Here,
$\widetilde{H}^{\sigma,p,(\mu)}(\Omega,\varrho)$ stands for the  Roitberg--Sobolev space with the smoothness index $\sigma\in\mathbb{R}$, integrability exponent $p\in(1,\infty)$, specific integer-valued index $\mu\geq0$, the norm in this space depending on the parameter $\varrho>0$. If $\sigma\geq\max\{\mu-1+1/p,0\}$, then $\widetilde{H}^{\sigma,p,(\mu)}(\Omega,\varrho)=
H^{\sigma}_{p}(\Omega,\varrho)$ with the equivalence of norms in which constants do not depend on $\varrho\geq1$ \cite[Section~9.1.2, p.~298]{Roitberg96}. Roitberg \cite[Section~1.13]{Roitberg96} defines parameter-dependent norms in $H^{\sigma}_{p}(\Omega)$ and $B^{\sigma}_{p}(\Gamma)$ in a different way then that we use; however, these norms are equivalent to ours with constants not depending on the parameter provided that $\sigma>0$ (see  Remark~\ref{Roitberg-parameter-spaces} given below). Thus, the above-mentioned result \cite[Theorem 9.1.1]{Roitberg96} contains assertion (ii) in the  $\varphi(\cdot)\equiv1$ case because $s>\max\{2q,m+1/p\}$ by the hypothesis.

If $p=2$, then assertion (ii) for every $\varphi\in\Upsilon$ is deduced  from the $\varphi(\cdot)\equiv1$ case by the quadratic interpolation on the base of \cite[Lemma~4.1]{MikhailetsMurach14} (which is a version of Lemma~\ref{lemma4.2} for this interpolation) in the same way as  \cite[Proof of Theorem~4.9]{MikhailetsMurach14} (which deals with the case where $m\leq2q-1$).

If $p\neq2$, then assertion (ii) for every $\varphi\in\Upsilon$ is deduced on the base of Lemma~\ref{lemma4.2}. Consider first the case where $p\neq2$ and $2q>m+1/p$. Let us choose a number $p_{1}$ such that $1<p_{1}<p<2$ or $2<p<p_{1}<\infty$ and define parameters $\theta\in(0,1)$ and $\varphi_{0}\in\Upsilon$ according to Theorem~\ref{th-C-interp-Sobolev-got}. If $1<p_{1}<p<2$, we choose a number $s_1$ such that $2q<s_1<s$ and then define a number $s_0>s$ by the formula $s=(1-\theta)s_{0}+\theta s_{1}$. If  $2<p<p_{1}$, we choose a number $s_0$ such that $2q<s_0<s$ and then define a number $s_1>s$ by the same formula. By assertion~(i), we have the isomorphisms
\begin{equation}\label{isomorphism-Sobolev-classic}
(A(\lambda),B(\lambda)):H^{s_0,\varphi_0}_{2}(\Omega,|\lambda|)
\leftrightarrow\Xi^{s_0-2q,\varphi_0}_{2}(\Omega,\Gamma,|\lambda|)
\end{equation}
and
\begin{equation}\label{isomorphism-Sobolev-generalized}
(A(\lambda),B(\lambda)):H^{s_1}_{p_{1}}(\Omega,|\lambda|)\leftrightarrow
\Xi^{s_1-2q}_{p_{1}}(\Omega,\Gamma,|\lambda|)
\end{equation}
for every $\lambda\in K$ with $|\lambda|\geq\max\{\lambda_{0},1\}$. Here, \begin{equation*}
\Xi^{s-2q,\varphi}_{p}(\Omega,\Gamma,\varrho):=
H^{s-2q,\varphi}_{p}(\Omega,\varrho)\times
\prod_{j=1}^{q}B^{s-m_{j}-1/p,\varphi}_{p}(\Gamma,\varrho)
\end{equation*}
whenever $p\in(1,\infty)$, $s>\max\{2q,m+1/p\}$, $\varphi\in\Upsilon$, and $\varrho\geq1$. As we have indicated in the proof, these isomorphisms satisfy \eqref{4.28}, which means that the norms of operators \eqref{isomorphism-Sobolev-classic} and \eqref{isomorphism-Sobolev-generalized} and the norms of their inverses are bounded from above by a number $c>0$ that does not depend on $\lambda$. One of isomorphisms \eqref{isomorphism-Sobolev-classic} and \eqref{isomorphism-Sobolev-generalized} is an extension by the other due to our choice of $s_0$ and $s_1$. Hence, interpolating between \eqref{isomorphism-Sobolev-classic} and \eqref{isomorphism-Sobolev-generalized}, we obtain another isomorphism
\begin{equation}\label{isomorphism-got}
(A(\lambda),B(\lambda)):\bigl[H^{s,\varphi_0}_{2}(\Omega,|\lambda|),
H^{s}_{p_{1}}(\Omega,|\lambda|)\bigr]_{\theta}\leftrightarrow
\bigl[\Xi^{s_0-2q,\varphi_0}_{2}(\Omega,\Gamma,|\lambda|),
\Xi^{s_1-2q}_{p_{1}}(\Omega,\Gamma,|\lambda|)\bigr]_{\theta}
\end{equation}
for every $\lambda\in K$ with $|\lambda|\geq\max\{\lambda_{0},1\}$. Since the interpolation functor $[\cdot,\cdot]_{\theta}$ is exact of type $\theta$, the norms of operator \eqref{isomorphism-got} and its inverse are bounded from above by the same number~$c$. It follows from this by Lemma~\ref{lemma4.2} that the two-sided estimate \eqref{4.28} holds true with constants that do not depend on $\lambda$.

The case where $p\neq2$ and $2q<m+1/p$ is treated in the same way as Theorem~\ref{th-Fredholm} provided that we use Lemma~\ref{lemma4.2} when interpolating between isomorphisms.
\end{proof}

\begin{remark}
The right-hand side of the bilateral estimate \eqref{4.28} holds true for  all $\lambda\in\mathbb{C}$ with $|\lambda|\geq\max\{\lambda_{0},1\}$ without the assumption that the boundary-value problem \eqref{4.24}, \eqref{4.25} is parameter-elliptic in~$K$. This is known in the $\varphi(\cdot)\equiv1$ case (see, e.g., \cite[Lemma~9.1.1]{Roitberg96}) and is therefore substantiated for every $\varphi\in\Upsilon$ in the proof of assertion~(ii).
\end{remark}

\begin{remark}\label{Roitberg-parameter-spaces}
Let $s\in\mathbb{R}$, $p\in(1,\infty)$, and $1\leq n\in\mathbb{Z}$. Roitberg \cite[Subsections 1.13.1 and 1.13.3]{Roitberg96} defines equivalent parameter-dependent norms in the spaces $H^{s}_{p}(\mathbb{R}^{n})$ and $B^{s}_{p}(\mathbb{R}^{n})$ as follows:
\begin{equation*}
\|w,H^{s}_{p}(\mathbb{R}^{n})\|_{\varrho}:=
\|\mathcal{F}^{-1}[(1+|\xi|^{2}+\varrho^{2})^{s/2}(\mathcal{F}w)(\xi)],
L_{p}(\mathbb{R}^{n})\|,
\end{equation*}
where $w\in H^{s}_{p}(\mathbb{R}^{n})$, and
\begin{equation*}
\|w,B^{s}_{p}(\mathbb{R}^{n})\|_{\varrho}:=
\|\mathcal{F}^{-1}[(1+|\xi|^{2}+\varrho^{2})^{s/2}(\mathcal{F}w)(\xi)],
B^{0}_{p}(\mathbb{R}^{n})\|,
\end{equation*}
where $w\in B^{s}_{p}(\mathbb{R}^{n})$; here, $\varrho$ is a positive parameter. These norms induce equivalent para\-me\-ter-depen\-dent norms in the spaces $H^{s}_{p}(\Omega)$, with $s\geq0$, and $B^{s}_{p}(\Gamma)$, with $s\in\mathbb{R}$, in the same way as  \eqref{Sobolev-norm-domain} and \eqref{Besov-norm-boundary}, resp \cite[Subsections 1.13.5 and 1.13.6]{Roitberg96}.
If $s>0$, then there exists a number $c\geq1$ such that
\begin{equation}\label{equivalence-norms-parameter}
c^{-1}\|w,E^{s}_{p}(\mathbb{R}^{n},\varrho)\|\leq
\|w,E^{s}_{p}(\mathbb{R}^{n})\|_{\varrho}\leq c\,
\|w,E^{s}_{p}(\mathbb{R}^{n},\varrho)\|
\end{equation}
for all $w\in E^{s}_{p}(\mathbb{R}^{n})$ and $\varrho\geq1$; here, $c$ does not depend on $w$ and $\varrho$, whereas $E$ denotes either $H$ or $B$. This is shown in \cite[p.~170, formula~(1.6)]{GrubbKokholm93} for $E=H$ (Sobolev spaces) and is analogously proved for $E=B$ (Besov spaces) with the help of Mihlin's criterion for a function to be a Fourier multiplier in $L_{p}(\mathbb{R}^{n})$ (and hence in Sobolev spaces and then in Besov spaces by the real interpolation between Sobolev spaces). Of course, property \eqref{equivalence-norms-parameter} implies its analogs for the parameter-dependent norms in $H^{s}_{p}(\Omega)$ and $B^{s}_{p}(\Gamma)$.
\end{remark}

\appendix

\section{}\label{appendix}

We recall the definition and some properties of the quadratic interpolation with function parameter between Hilbert spaces. This interpolation is systematically used in the paper together with the well-known complex interpolation (with number parameter) between Banach spaces. Such a quadratic interpolation was introduced by Foia\c{s} and Lions \cite[p.~278]{FoiasLions61}. Expounding it, we mainly follow the monograph \cite[Section~1.1]{MikhailetsMurach14}. It is sufficient for our purposes to restrict ourselves to separable Hilbert spaces.

Let $E_{0}$ and $E_{1}$ be separable complex Hilbert spaces such that $E_{1}$ is a dense linear manifold in $E_{0}$ and that $\|w,E_{0}\|\leq c\,\|w,E_{1}\|$ for every vector $w\in E_{1}$, with the number $c>0$ not depending on~$w$ (in other words, the continuous dense embedding $E_{1}\hookrightarrow E_{0}$ holds true). The ordered pair $(E_{0},E_{1})$ of such spaces is called regular. For it, there exists a unique positive-definite self-adjoint operator $J$ given in the Hilbert space $E_{0}$ and such that $E_{1}$ is the domain of $J$ and that $\|Jw,E_{0}\|=\|w,E_{1}\|$ for every $w\in E_{1}$.

Let $\mathcal{B}$ denote the set of all Borel measurable functions
$\psi:(0,\infty)\rightarrow(0,\infty)$ such that $\psi$ is bounded on each compact interval $[a,b]$, with $0<a<b<\infty$, and that $1/\psi$ is bounded on every set $[r,\infty)$, with $r>0$. Given $\psi\in\mathcal{B}$ and applying the spectral theorem to the self-adjoint operator $J$, we obtain the (generally, unbounded) operator $\psi(J)$ in $E_{0}$. Let $(E_{0},E_{1})_{\psi}$ denote the domain of $\psi(J)$ endowed with the norm $\|u,(E_{0},E_{1})_{\psi}\|:=\|\psi(J)u,E_{0}\|$. The space $(E_{0},E_{1})_{\psi}$ is Hilbert and separable with respect to this norm and is continuously embedded in $E_0$.

We call a function $\psi\in\mathcal{B}$ an interpolation parameter if the following condition is satisfied for all regular pairs $(E_{0},E_{1})$ and $(G_{0},G_{1})$ of Hilbert spaces and for an arbitrary linear mapping $T$ given on whole $E_{0}$: if
\begin{equation}\label{restriction-T-bounded}
\mbox{the restriction of $T$ to $E_{j}$ is a bounded operator
$T:E_{j}\to G_{j}$ for each $j\in\{0,1\}$},
\end{equation}
then the restriction of $T$ to $(E_{0},E_{1})_{\psi}$ is a bounded operator from $(E_{0},E_{1})_{\psi}$ to $(G_{0},G_{1})_{\psi}$. We say in this case that $(E_{0},E_{1})_{\psi}$ is obtained by the quadratic interpolation with the function parameter $\psi$ between the spaces $E_{0}$ and $E_{1}$ and that $T:(E_{0},E_{1})_{\psi}\to (G_{0},G_{1})_{\psi}$ is the result of this interpolation applied to the operators $T:E_{j}\to G_{j}$ with $j\in\{0,1\}$.

A function $\psi\in\mathcal{B}$ is an interpolation parameter if and only if $\psi$ is pseudoconcave on the set $(r,\infty)$ for certain $r>0$, i.e. there exists a positive concave function $\psi_{1}(t)$ of $t\gg1$ such that
the functions $\psi/\psi_{1}$ and $\psi_{1}/\psi$ are bounded on this set. This fundamental fact follows from Peetre's \cite{Peetre68} description of all interpolation functions of positive order \cite[Theorem~1.9]{MikhailetsMurach14}. Specifically, the function $\psi(t):=t^{\sigma}\psi_{0}(t)$ of $t\geq1$ is an interpolation parameter provided that the number $\sigma$ satisfies $0<\sigma<1$ and that the function $\psi_{0}:[1,\infty)\to(0,\infty)$ is Borel measurable and varies slowly at infinity (in the sense of Karamata) \cite[Theorem~1.11]{MikhailetsMurach14}.

The norm of a linear operator obtained by the quadratic interpolation admits the following estimate \cite[Theorem~1.8]{MikhailetsMurach14}: for every interpolation parameter $\psi\in\mathcal{B}$ and each number $\nu>0$ there exists a number $c=c(\psi,\nu)>0$ such that
\begin{equation*}
\|T:(E_{0},E_{1})_{\psi}\to (G_{0},G_{1})_{\psi}\|\leq
c\,\max\bigl\{\|T:E_{j}\to G_{j}\|:j\in\{0,1\}\bigr\}.
\end{equation*}
Here, $(E_{0},E_{1})$ and $(G_{0},G_{1})$ are arbitrary regular pairs of Hilbert spaces for which the norms of the embedding operators $E_{1}\hookrightarrow E_{0}$ and $G_{1}\hookrightarrow G_{0}$ do not exceed $\nu$, whereas $T$ is an arbitrary linear mapping given on $E_{0}$ and satisfying \eqref{restriction-T-bounded}. This fact is used in the proof of assertion~(ii) of Theorem~\ref{th4.9} as well as the following estimate of the norm of an operator  obtained by the complex interpolation with a number parameter $\theta\in(0,1)$:
\begin{equation}\label{estimate-norms-C-interp}
\|T:[E_{0},E_{1}]_{\theta}\to[G_{0},G_{1}]_{\theta}\|\leq
\|T:E_{0}\to G_{0}\|^{1-\theta}\,\|T:E_{1}\to G_{1}\|^{\theta}.
\end{equation}
Here,  $(E_{0},E_{1})$ and $(G_{0},G_{1})$ are arbitrary interpolation pairs of Banach spaces, whereas $T$ is an arbitrary linear mapping acting from whole $E_{0}+E_{1}$ to $G_{0}+G_{1}$ and satisfying \eqref{restriction-T-bounded}. This estimate means that the interpolation functor $[\cdot,\cdot]_{\theta}$ is exact of type $\theta$ (see, e.g., \cite[Theorem 1.9.3(a) and Definition 1.2.2/2]{Triebel95}). This property is also used in the proof of Lemma~\ref{lemma4.2}.

As to interpolation of Fredholm bounded operators, we note that a version of Proposition~\ref{prop-interpolation-Fredholm} is valid for the quadratic interpolation with function parameter \cite[Theorem~1.7]{MikhailetsMurach14}.

\subsection*{Funding details.} The first named author was funded by the German Research Foundation (DFG), project 530831274 (https://gepris.dfg.de/gepris/projekt/530831274). The second named author was funded by the National Academy of Sciences of Ukraine, a grant from the Simons Foundation (1030291, A.A.M., and 1290607, A.A.M.), and Universities-for-Ukraine Non-residential Fellowship granted by UC Berkeley Economics/Haas. The second named author was also supported by the European Union's Horizon 2020 research and innovation programme under the Marie Sk{\l}odowska-Curie grant agreement No~873071 (SOMPATY: Spectral Optimization: From Mathematics to Physics and Advanced Technology).

\subsection*{Disclosure statement.} The authors report there are no competing interests to declare.

\end{document}